\documentclass[11pt]{article}

\usepackage[utf8]{inputenc}
\usepackage{fullpage}
\usepackage{marginnote}
\usepackage{graphicx}
\usepackage{array}
\usepackage{hyperref}
\usepackage{paralist}
\usepackage{amsthm}
\usepackage{amsmath}
\usepackage{amssymb}
\usepackage{color}
\usepackage{mathabx}

\title{A trinity of duality: non-separable planar maps, $\beta$-(1,0) trees and synchronized intervals}
\author{Wenjie Fang \\ Laboratoire de l'Informatique du Parall\'elisme, ENS de Lyon \\ wenjie.fang@ens-lyon.fr}

\newtheorem{thm}{Theorem}[section]
\newtheorem{prop}[thm]{Proposition}
\newtheorem{lem}[thm]{Lemma}
\newtheorem{coro}[thm]{Corollary}

\newcommand{\opname}[1]{\operatorname{\textup{\textsf{#1}}}}

\newcommand{\dtreeref}{\cite{JS1998bijective}}
\newcommand{\dsymref}{\cite{kitaev-involution}}
\newcommand{\dinvref}{\cite{kitaev-nsp}}
\newcommand{\tambijref}{\cite{tam-non-sep}}
\newcommand{\tutte}{\cite{Tutte:census}}
\newcommand{\kitaev}{\cite{kitaev-book}}
\newcommand{\dnsref}{\cite{kitaev2009decomp}}
\newcommand{\tamref}{\cite{PRV2014extension}}

\newcommand{\rootB}{\opname{root}}
\newcommand{\flT}{\opname{fl}}
\newcommand{\contI}{\opname{cont}}
\newcommand{\PiM}{\Pi_{\mathcal{M}}}
\newcommand{\PiB}{\Pi_{\mathcal{B}}}
\newcommand{\PiT}{\Pi_{\mathcal{T}}}
\newcommand{\PiS}{\Pi_{\mathcal{S}}}
\newcommand{\PiI}{\Pi_{\mathcal{I}}}
\newcommand{\DeltaM}{\Delta_{\mathcal{M}}}
\newcommand{\DeltaB}{\Delta_{\mathcal{B}}}
\newcommand{\DeltaT}{\Delta_{\mathcal{T}}}
\newcommand{\DeltaS}{\Delta_{\mathcal{S}}}
\newcommand{\DeltaI}{\Delta_{\mathcal{I}}}
\newcommand{\oplusM}{\oplus_{\mathcal{M}}}
\newcommand{\oplusB}{\oplus_{\mathcal{B}}}
\newcommand{\oplusT}{\oplus_{\mathcal{T}}}
\newcommand{\oplusS}{\oplus_{\mathcal{S}}}
\newcommand{\oplusI}{\oplus_{\mathcal{I}}}
\newcommand{\hdM}{\operatorname{hd}_{\mathcal{M}}}
\newcommand{\hdB}{\operatorname{hd}_{\mathcal{B}}}
\newcommand{\hdT}{\operatorname{hd}_{\mathcal{T}}}
\newcommand{\hdS}{\operatorname{hd}_{\mathcal{S}}}
\newcommand{\hdI}{\operatorname{hd}_{\mathcal{I}}}
\newcommand{\tlM}{\operatorname{tl}_{\mathcal{M}}}
\newcommand{\tlB}{\operatorname{tl}_{\mathcal{B}}}
\newcommand{\tlT}{\operatorname{tl}_{\mathcal{T}}}
\newcommand{\tlS}{\operatorname{tl}_{\mathcal{S}}}
\newcommand{\tlI}{\operatorname{tl}_{\mathcal{I}}}
\newcommand{\hT}{h_{\mathcal{T}}}
\newcommand{\rpath}{\opname{rpath}}
\newcommand{\dual}{\operatorname{dual}}
\newcommand{\mir}{\operatorname{mir}}
\newcommand{\fdeg}{\opname{fdeg}}

\newcommand{\tdef}[1]{\textcolor{blue}{\underline{{\smash{\textit{#1}}}}}}
\newcommand{\NR}{\operatorname{NR}}

\begin{document}
\maketitle

\abstract{
  The dual of a map is a fundamental construction on combinatorial maps, but many other combinatorial objects also possess their notion of duality. For instance, the Tamari lattice is isomorphic to its order dual, which induces an involution on the set of so-called ``synchronized intervals'' introduced by Pr\'eville-Ratelle and the present author. Another example is the class of $\beta$-(1,0) trees, which has a mysterious involution $h$ proposed by Claesson, Kitaev and Steingr\'imsson (2009). These two classes of combinatorial objects are all in bijection with the class of non-separable planar maps, which is closed by map duality. In this article, we show that we can identify the notions of duality in these three classes of objects using previously known natural bijections, which leads to a bijective proof of a result from Kitaev and de Mier (2013). We also discuss how various statistics are transferred by duality and by the bijections we used.
}

\section{Introduction}

Combinatorial maps (or \emph{maps} for short) are sometimes found to be in bijection with surprisingly many classes of combinatorial objects of diverse flavors. For instance, non-separable planar maps were found to be in bijection with so-called $\beta(1,0)$-trees by Cori, Jacquard and Schaeffer \cite{desc-tree, JS1998bijective}, which are in bijection with several classes of pattern-avoiding permutations (\textit{cf.} \cite[Chapter~2.2]{kitaev-book}). More recently, in \tambijref{}, a bijection between non-separable planar maps and intervals in generalized Tamari lattices was discovered. These intervals were first defined in \tamref{}, and were implicitly showed to be in bijection with a special kind of intervals in the classic Tamari lattice called ``synchronized intervals''.

In the study of maps, map duality is an important concept. It was first observed by Tutte \cite{Tutte:census} that the dual of a non-separable planar map is still non-separable, meaning that map duality is an involution in the class of non-separable planar maps. Other notions of duality can also be defined on combinatorial classes we mentioned above. For instance, given a partial order, its order dual is the reversed order relation. When defined on binary trees, it is well-known that the Tamari lattice is isomorphic to its order dual, and such an isomorphism can be given by taking the mirror image of binary trees, which induces an involution on synchronized intervals (implicitly proved in \tamref{}). For $\beta(1,0)$-trees, there is also a mysterious involution $h$ that was proposed in \cite{kitaev2009decomp} and proved to be an involution in \dsymref{}. These dualities induce equi-distribution results of various statistics (\textit{cf.} \cite{kitaev-book}).

In this paper, we will provide an integrated point of view on the duality of these three classes of objects. We show that the three dualities mentioned above are in fact related by natural bijections, thus can be regarded as the same duality in their recursive structure seen through the prism of different families of objects. To express our result more precisely, we introduce another class of objects called \emph{decorated trees}, first defined in \tambijref{}, in which they relay a natural bijection from non-separable planar maps to synchronized intervals. We now denote by $\textrm{dual}$ the map duality, $\textrm{mir}$ the duality on synchronized intervals, $h$ the involution of $\beta(1,0)$-trees in \dsymref{}, and $h_\mathcal{T}$ its counterpart in decorated trees. We prove that the diagram below commutes.

\begin{center}
  \includegraphics[page=21,scale=0.9]{figure-inv.pdf}
\end{center}

Here, $\textrm{T}$ and $\textrm{I}$ are natural bijections first defined in \tambijref{}, and $\varphi_T$ is a natural bijection from decorated trees to $\beta(1,0)$-trees that will be defined later. As a consequence, we give an alternative proof of results in \dsymref{}, which states that the number of fixed points of $h$ among $\beta(1,0)$-trees with $n$ edges is equal to the number of self-dual non-separable planar maps with $n+1$ edges. The number of such maps was then given in \cite{count-self-dual}. In \dsymref{}, the authors asked for a bijective explanation, to which our proof responds. Another consequence is that the isomorphisms between different generalized Tamari lattices proved in \tamref{} can now be understood via the more intuitive map duality. We then also discuss statistics that are transferred by these bijections and dualities. Our results thus provide an integrated view of the structures of all three classes of combinatorial objects, which can be used to transfer tools in the study of one class to another and benefits the structural study of these objects.

To get the results mentioned above, we investigated recursive decompositions of all four classes of objects involved. We exhibit for each class of objects a recursive decomposition, and show that all decompositions are isomorphic to each other. In particular, for non-separable planar maps, we used here the so-called \emph{parallel decomposition}, which contrasts with the \emph{series decomposition} used heavily in previous work, such as \dtreeref{}. We then show that the bijections $\mathrm{T}, \mathrm{I}, \varphi_{\mathcal{T}}$ are all \emph{canonical bijections} with respect to these recursive decompositions, albeit the fact that they were originally defined as direct bijections. As a byproduct, we also analyze the relation between parallel and series decompositions of non-separable planar maps, and we find that they are essentially related by map duality.

Our results here are motivated by a recent trend on the bijective understanding of the relations between planar maps and intervals in Tamari-like lattices. It was first discovered by Chapoton in \cite{ch06} using generating series that the number of intervals in the Tamari lattice of order $n$ coincides with that of planar 3-connected triangulations with $3(n+1)$ edges. The same result was then proved bijectively by Bernardi and Bonichon in \cite{BB2009intervals}. Intervals in a generalization of the Tamari lattice, called the $m$-Tamari lattice, was also enumerated by Bousquet-Mélou, Fusy and Préville-Ratelle in \cite{bousquet-fusy-preville}, and the formula they found are similar to those of planar maps. This resemblance is not yet fully explained. More recently, Viennot and Préville-Ratelle proposed the generalized Tamari lattice in \cite{PRV2014extension}, and Préville-Ratelle and the present author proved intervals in all the generalized Tamari lattices of length $n$ are in bijection with non-separable planar maps with $n+2$ edges. Our result is a more fine-grained study on these relations between planar maps and intervals in Tamari-like lattices. There are other works (see \cite{chapoton-chatel-pons, interval-poset}) that relate Tamari intervals to combinatorial objects other than planar maps, which also provide insights to the structure of these intervals. Other than intervals in Tamari-like lattices, non-separable planar maps are also related to some families of pattern-avoiding permutations, which can be traced back to the work of Dulucq, Gire and West \cite{dulucq1993permutations}. This relation also motivated the study of $\beta(1,0)$-trees \dnsref{}. We hope that the structural relations we get here can lead to a transfer of results and methods between non-separable planar maps, generalized Tamari intervals and pattern-avoiding permutations.

The structure of this paper is as follows. In Section~\ref{sec:decomp}, definitions of the concerned combinatorial classes are laid out, alongside with a recursive decomposition and a notion of duality of each class. In Section~\ref{sec:nat-bij}, we prove that the natural bijections $\mathrm{T}$ and $\mathrm{I}$ defined in \tambijref{} are in fact also canonical bijections with respect to the previously stated recursive decompositions. In Section~\ref{sec:inv}, we prove that the three notions of duality in $\beta(1,0)$-trees, non-separable planar maps and Tamari intervals are the same in the light of the canonical bijections. In Section~\ref{sec:series}, we take a small detour about how another recursive decomposition of non-separable planar maps is related to the previously stated one by duality. We end this article by a discussion about how various statistics are transferred under our bijections.

\section{Recursive decompositions of related objects} \label{sec:decomp}

\subsection{Description trees}

In \dtreeref{}, to study non-separable planar maps, Jacquard and Schaeffer propose a combinatorial object that was later called $\beta(1,0)$-tree in \dnsref{}. We now pick up the formalism from \dsymref{} and \dinvref{}. A \tdef{$\beta$-(1,0) tree} is a rooted plane tree with labels on its nodes (internal ones and leaves) such that all leaves are labeled $1$, the label of the root is the sum of the labels of its children, and for other internal nodes, their labels are between $1$ and the sum of the labels of their children. Figure~\ref{fig:tree-case} gives an example of a $\beta$-(1,0) tree.

Let $\mathcal{B}_n$ be the set of $\beta$-(1,0) trees with $n$ edges, and $\mathcal{B} = \cup_{n \geq 1} \mathcal{B}_n$ the set of all $\beta$-(1,0) trees. We denote by $\opname{root}(B)$ the label of the root of a $\beta$-(1,0) tree $B$.

We now consider the following recursive construction of $\beta$-(1,0) trees. Let $B$ be a $\beta$-(1,0) tree and $v$ its root, it must be in exactly one of the following four cases (referred later as \tdef{types} of trees), as in Fig.~\ref{fig:tree-case}.
\begin{itemize}
\item \textbf{Type I}: The root $v$ has only one child, and it is a leaf.
\item \textbf{Type II}: The root $v$ has only one child, and it is an internal node.
\item \textbf{Type III}: The root $v$ has at least two children, and the leftmost one is a leaf.
\item \textbf{Type IV}: The root $v$ has at least two children, and the leftmost one is an internal node.
\end{itemize}

\begin{figure}
\centering
\includegraphics[page=8, scale=0.6]{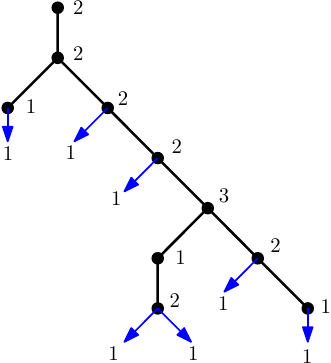} \quad \quad \includegraphics[page=9]{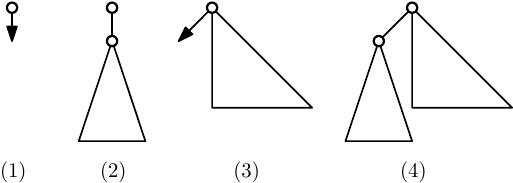}
\caption{A $\beta$-(1,0) tree and four types of trees} \label{fig:tree-case}
\end{figure}

We now define some operations on $\beta$-(1,0) trees that will be used in the recursive decomposition. Let $B$ be a $\beta$-(1,0) tree and $k = \rootB(B)$. For any $1 \leq i \leq k$, we denote by $\DeltaB(B,i)$ the tree obtained by attaching the root of $B$ to a new vertex as a child and label both the old and the new root with $i$. By definition, $\DeltaB(B,i) \in \mathcal{B}$. If $B$ is of type II, we define $\PiB(B)$ as the tree obtained by removing the root and adjusting the label of the new root to the sum of labels of its children. The actions of $\DeltaB$ and $\PiB$ are illustrated on the left part of Figure~\ref{fig:dtree-fct}. The following property is clear. 

\begin{prop}\label{prop:delta-pi-b}
Let $B$ be a $\beta$-(1,0) tree, for any $1 \leq i \leq \rootB(B)$, we have $\PiB(\DeltaB(B,i)) = B$ and $\rootB(\DeltaB(B,i))=i$. Moreover, for any $\beta$-(1,0) tree $B'$ of type II such that $\PiB(B') = B$, there is a unique integer $j$ between $1$ and $\rootB(B)$ such that $B' = \DeltaB(B,j)$. 
\end{prop}

\begin{figure}
\centering
\includegraphics[page=10]{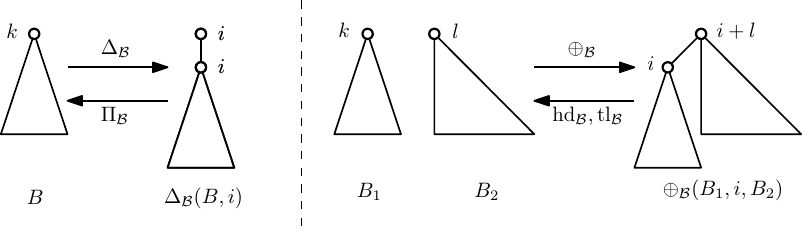}
\caption{Functions $\DeltaB, \PiB, \oplusB, \hdB$ and $\tlB$ over $\beta$-(1,0) trees} \label{fig:dtree-fct}
\end{figure}

Let $B_1, B_2$ be two $\beta$-(1,0) trees, $k = \rootB(B_1)$ and $1 \leq i \leq k$. We denote by $\oplusB(B_1, i, B_2)$ the tree obtained by identifying the two roots of $\DeltaB(B_1,i)$ and $B_2$, with the original root of $B_1$ as the leftmost child of the new root. The new root is then labeled by the sum of labels of its children. By definition, we have $\oplusB(B_1,i,B_2) \in \mathcal{B}$. Conversely, given a $\beta$-(1,0) tree $B$ of type IV, we denote by $\hdB(B)$ the sub-tree rooted at the leftmost child of the root, with root label adjusted, and $\tlB(B)$ the other parts of the tree with root label adjusted. From definition we know that $\hdB(B)$ and $\tlB(B)$ are both $\beta$-(1,0) trees. These functions are illustrated in the right part of Figure~\ref{fig:dtree-fct}. The names $\hdB$ and $\tlB$ is taken from functional programming terminology, meaning ``head'' and ``tail'' respectively of a list. The following properties of $\oplusB$, $\hdB$ and $\tlB$ are also clear.
\begin{prop}\label{prop:oplus-hd-tl-b}
For two $\beta$-(1,0) trees $B_1, B_2$ and an integer $1 \leq i \leq \rootB(B_1)$, we have $\hdB(\oplusB(B_1, i, B_2)) = B_1$, $\tlB(\oplusB(B_1, i, B_2)) = B_2$ and $\rootB(\oplusB(B_1, i, B_2)) = i + \rootB(B_2)$.

Furthermore, for any $\beta$-(1,0) tree $B'$ of type IV, if $\hdB(B') = B_1$ and $\tlB(B') = B_2$, then there is a unique integer $j$ between $1$ and $\rootB(B_1)$ such that $B' = \oplusB(B_1, j, B_2)$.
\end{prop}

We now have the following recursive decomposition of $\beta$-(1,0) trees according to tree types.

\begin{itemize}
\item \textbf{Type I}: This is the base case of the recursive decomposition.
\item \textbf{Type II}: In this case, the function $\PiB$ applies. By Proposition~\ref{prop:delta-pi-b}, we can get all trees in this case once and only once using $\DeltaB$.
\item \textbf{Type III}: The leftmost child $\ell$ of the root $v$ is a leaf. By removing $\ell$ and decreasing the label of $v$ by $1$, we get a smaller $\beta$-(1,0) tree. We can get any $\beta$-(1,0) tree in this case by first taking an arbitrary $\beta$-(1,0) tree $B'$, then attaching a leaf to the root of $B'$ and the leftmost child and adding $1$ to the root label.
\item \textbf{Type IV}: In this case, the function $\hdB$ and $\tlB$ apply. By Proposition~\ref{prop:oplus-hd-tl-b}, we can get all trees in this case once and only once using $\oplusB$.
\end{itemize}

\subsection{Non-separable planar maps}

We now turn to non-separable planar maps. A \tdef{planar map} is an embedding of an undirected graph into the plane, with one distinguished and oriented edge adjacent to the infinite face called the \tdef{root}. As a convention, the infinite face of a planar map is called the \tdef{outer face}, and it should be on the left of the root. We call the face on the right of the root the \tdef{core face}. The starting vertex of the root is called the \tdef{root vertex}. A planar map is called \tdef{separable} if its edges can be partitioned into two sets $S,T$ such that exactly one vertex $v$ is adjacent to edges of both sets. In this case, the vertex $v$ is called a \tdef{cut vertex}. A \tdef{non-separable planar map} (or simply \tdef{NSP-map}) is a planar map with at least two edges that is not separable. We exclude the two one-edges maps here. Figure~\ref{fig:map} gives an example of a non-separable planar map.

\begin{figure}
\centering
\includegraphics[page=1, scale=0.7]{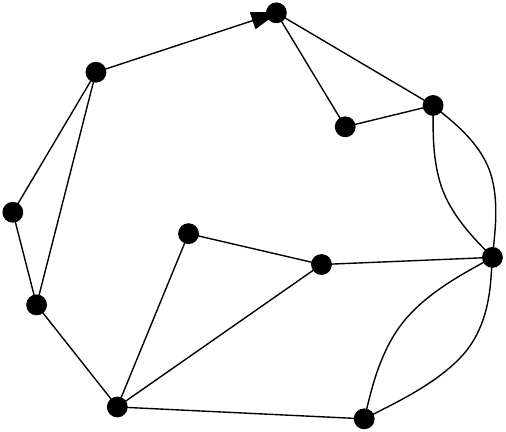} \quad \quad \quad \includegraphics[page=5,scale=0.7]{figure-inv.pdf}
\caption{A non-separable planar map and its dual} \label{fig:map}
\end{figure}

We now consider the duality on NSP-maps. The \tdef{dual} of a planar map $M$, denoted by $\dual(M)$, is the map obtained by the following process: we first put a vertex $v_f$ inside each face $f$ of $M$, then for each edge $e$ adjacent to two (not necessarily distinct) faces $f_1,f_2$, we put an edge $e^*$ that links $v_{f_1}$ and $v_{f_2}$ which only crosses $e$ among edges in $M$ and only once. We say that $v_f$ is the \tdef{dual vertex} of $f$, and $f$ the \tdef{dual face} of $v_f$. Furthermore, the root vertex of $\dual(M)$ is the dual vertex of the outer face of $M$, and the outer face of $\dual(M)$ is the dual face of the root vertex of $M$. This definition of duality is also illustrated in Figure~\ref{fig:map}. We notice that our notion of duality is exactly the same as that in \cite{count-self-dual}. The following property was proved by Tutte in \tutte{}.

\begin{prop}[\tutte{}] \label{prop:non-sep-self-dual}
Let $M$ be an NSP-map, then $\dual(M)$ is also an NSP-map.
\end{prop}

Let $\mathcal{M}_n$ be the set of NSP-maps containing $n+1$ edges, and $\mathcal{M} = \bigcup_{n \geq 1} \mathcal{M}_n$ the set of all NSP-maps. For $M$ an NSP-map and $v$ its root vertex, we define the statistics $\opname{deg}(M) = \opname{deg}(v) - 1$ as the degree of its root vertex \emph{minus} 1.

We now consider a recursive construction of NSP-maps that we call the \tdef{parallel decomposition}. Let $M$ be an NSP-map and $e=(v \to u)$ its root. By contracting $e$ to a new vertex $v'$, we get another planar map $M'$ that may be separable or contain loops. However, the only possible cut vertex of $M'$ is $v'$, since a cut vertex $w' \neq v'$ in $M'$ must also corresponds to a cut vertex in $M$. By cutting $v'$ in $M'$, we may have several non-separable components. We observe that, since $M$ is non-separable, each component must be adjacent to both $v$ and $u$ in $M$, we thus have all components separated by faces that are adjacent to both $u$ and $v$ in $M$, giving $M$ an onion-like structure. Among the components of $M'$, the \tdef{innermost component} is the one containing the core face.

The map $M$ must fall in exactly one of the following four cases (referred later as \tdef{parallel types} of NSP-maps), as illustrated in Figure~\ref{fig:map-case}.
\begin{itemize}
\item \textbf{Type I}: $M'$ has only one component, and it is a loop.
\item \textbf{Type II}: $M'$ has only one component, and it is not a loop.
\item \textbf{Type III}: $M'$ has at least two components, with the innermost one a loop.
\item \textbf{Type IV}: $M'$ has at least two components, with the innermost one not a loop.
\end{itemize}

\begin{figure}
\centering
\includegraphics[page=11, scale=0.8]{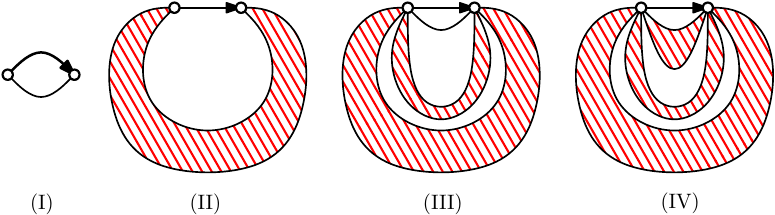}
\caption{Four cases in the parallel decomposition of NSP-maps} \label{fig:map-case}
\end{figure}

As a remark, previous studies on NSP-maps (\textit{c.f.} \cite{desc-tree, JS1998bijective}) usually use another decomposition that we call the \tdef{series decomposition} that decompose an NSP-map by deleting the root and investigating the non-separable components, which are now linked by cut vertices one by one like sausages. Figure~\ref{fig:map-decomp-compare} gives a visual comparison of these two decompositions. We will see that our parallel approach is better adapted to natural bijections. The relation between series and parallel decompositions will be discussed later.

\begin{figure}
\centering
\includegraphics[page=3, scale=0.8]{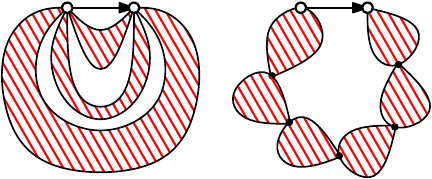}
\caption{Comparison of parallel and series decomposition of NSP-maps} \label{fig:map-decomp-compare}
\end{figure}

We now define some operations on NSP-maps that will be used in the parallel decomposition. Let $M$ be an NSP-map with the root $e=(v \to u)$ (with an orientation from $v$ to $u$). We denote by $\PiM(M)$ the map obtained by contracting $e$ and rooting at the next edge $e'$ on the outer face in clockwise order. We see that $\PiM(M)$ is non-separable if and only if $M$ is of type II. For any $1 \leq i \leq \opname{deg}(M)$, we denote by $\DeltaM(M,i)$ the map obtained by splitting the root vertex $v$ along the outer face into two new vertices $u$ and $v$, with $u$ adjacent to $e$ and $v$ of degree $i$, then adding an edge $e'$ pointing from $v$ to $u$ as the new root. The left part of Figure~\ref{fig:map-fct} illustrates the definition of these two functions. We have the following properties of these functions.

\begin{prop} \label{prop:pi-delta-m}
Let $M$ be a NSP-map, $i$ be an integer between $1$ and $\opname{deg}(M)$ and $M' = \DeltaM(M,i)$. We have $\opname{deg}(M') = i$ and $\PiM(M') = M$, and $M'$ is also an NSP-map. Furthermore, for any NSP-map $N$ such that $\PiM(N) = M$, there exists a unique integer $j$ between $1$ and $\opname{deg}(M)$ such that $N = \DeltaM(M,j)$.
\end{prop}
\begin{proof}
From the definition of $\DeltaM$ and $\PiM$, it is clear that $\opname{deg}(M')=i)$ and $\PiM(M')=M$. To prove that $M' \in \mathcal{M}$, let $e' = \{ u', v' \}$ be the new edge introduced in $M'$ that contracts to a vertex $v$ in $M$. Any partition of edges in $M'$ must take the form $S \cup \{ e \}, T$, where $S,T$ is a partition of edges in $M$. Let $V$ be the set of vertices of $M$ adjacent to edges both in $S$ and $T$, and $V'$ the same set for $S \cup \{ e \}$ and $T$ in $M'$. Since $M$ is an NSP-map, $V$ has at least two elements. If $v \notin V$, then $V \subset V'$; if $v \in V$, then one of $u',v'$ must be in $V'$, thus $|V'| \geq |V|$. Either case implies that $V'$ has at least two elements, thus $M'$ is an NSP-map.

For the second point, let $e'$ be the root of $N$. Since $\PiM(N) = M$, we can get $M$ from $N$ by contracting $e'$ to a vertex $v$ in $M$. Therefore, by splitting $v$ into two vertices linked by an edge $e$, we obtain $N$. By planarity and the fact that the outer face remains invariant, the way of splitting $v$ is totally determined by the degree of $v$ after splitting, thus exhausted by all $\DeltaM(M,j)$ with $1 \leq j \leq \opname{deg}(M)$, which implies that $N$ must also be in this form.
\end{proof}

\begin{figure}
\centering
\includegraphics[page=12, width=\textwidth]{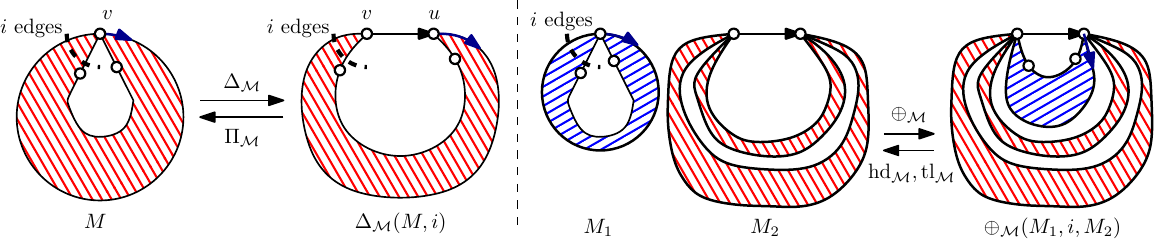}
\caption{Functions $\DeltaM, \PiM, \oplusM, \hdM$ and $\tlM$ over NSP-maps} \label{fig:map-fct}
\end{figure}

Let $M_1, M_2$ be two NSP-maps, and $1 \leq i \leq \opname{deg}(M_1)$. We denote by $\oplusM(M_1,i,M_2)$ the map obtained by putting $\DeltaM(M_1,i)$ in the core face of $B_2$ and identifying the two roots. Conversely, given an NSP-map $M$ of type IV, we define $\hdM(M)$ to be the innermost non-separable component of $\PiM(M)$ (with the two vertices of the root of $M$ identified), and $\tlM(M)$ the map that remains after removing $\hdM(M)$ from $M$. Both $\hdM(M)$ and $\tlM(M)$ are NSP-maps. The right part of Figure~\ref{fig:map-fct} illustrates the definitions of $\oplusM, \hdM$ and $\tlM$. We have the following properties.

\begin{prop} \label{prop:oplus-hd-tl-m}
For two NSP-maps $M_1, M_2$ and an integer $i$ between $1$ and $\opname{deg}(M_1)$, the map $M = \oplusM(M_1, i, M_2)$ is also an NSP-map, and $\hdM(M) = M_1$, $\tlM(M) = M_2$ and $\opname{deg}(M) = i + \opname{deg}(M_2)$.

Furthermore, for any NSP-map $M'$ of type IV, if $\hdM(M') = M_1$ and $\tlM(M') = M_2$, then there exists a unique integer $j$ between $1$ and $\opname{deg}(M_1)$ such that $M' = \oplusM(M_1, j, M_2)$.
\end{prop}
\begin{proof}
We first prove that $M \in \mathcal{M}$. Let $S,T$ be a partition of edges in $M$, and $e=\{ u, v \}$ its root. By construction of $M$, we know that all edges in $M$ originate either from $M_1$ or from $M_2$. We consider the set of edges $E_2$ originating from $M_2$. If neither $S \cap E_2$ nor $T \cap E_2$ is empty, since $M_2$ is an NSP-map, $S \cap E_2$ and $T \cap E_2$ must have at least two common adjacent vertices, thus $S$ and $T$ as well. The case where neither $S \cap E_1$ nor $T \cap E_1$ can be treated similarly. The only case left to check is $S=E_1$ and $T=E_2$, or symmetrically $S=E_2$ and $T=E_1$, but in this case both $u$ and $v$ are adjacent to $S$ and $T$. Therefore, for any partition $S,T$, there are at least two vertices adjacent to both edges in $S$ and in $T$. We conclude that $M$ is an NSP-map. The equalities follow easily from the construction.

For the second point, by duplicating the root, we can separate $M'$ into two maps, one isomorphic to $M_2$, the other, denoted by $M_1'$, corresponds to the rest. By the definition of $\hdM$, we have $\PiM(M_1') = \hdM(M') = M_1$. We then conclude by Proposition~\ref{prop:pi-delta-m}.
\end{proof}

We now consider the parallel decomposition in each case.
\begin{itemize}
\item \textbf{Type I}: This is the base case, which is a map of a pair of double edges.
\item \textbf{Type II}: In this case, the function $\PiM$ applies. By Proposition~\ref{prop:pi-delta-m}, all NSP-maps in this case can be obtained bijectively by applying $\DeltaM$ to NSP-maps with all possible extra parameters.
\item \textbf{Type III}: Let $e_1$ be the edge in $M$ corresponding to the loop. By deleting $e_1$ in $M$, we get a smaller NSP-map. To obtain an NSP-map $M$ in this case, we take an arbitrary NSP-map $M'$ and duplicate its root, and we have $\opname{deg}(M) = \opname{deg}(M') + 1$.
\item \textbf{Type IV}: In this case, the functions $\hdM$ and $\tlM$ apply. By Proposition~\ref{prop:oplus-hd-tl-m}, all NSP-maps in this case can be obtained bijectively by applying $\oplusM$ to NSP-maps with all possible extra parameters.
\end{itemize}

This recursive construction of NSP-maps is isomorphic to that of $\beta$-(1,0) tree, their size parameters coincide, and the statistics $\opname{deg}$ and $\opname{root}$ coincide too. We define the \tdef{canonical bijection} $\phi_M$ from the set of NSP-map $\mathcal{M}$ to the set of $\beta$-(1,0) trees $\mathcal{B}$ by the following recursive definition on the four types of NSP-maps. Let $M$ be an NSP-map, we define recursively $\phi_M(M)$ according to its type as follows.

\begin{itemize}
\item \textbf{Type I}: $\phi_M(M)$ is the type I $\beta$-(1,0) tree.
\item \textbf{Type II}: We have $M = \DeltaM(M', i)$, and we define $\phi_M(M) = \DeltaB(\phi_M(M'),i)$.
\item \textbf{Type III}: Let $M'$ be the map obtained by deleting the edge corresponding to the loop. We define $\phi_M(M)$ be the type III $\beta$-(1,0) tree obtained from $\phi_M(M')$.
\item \textbf{Type IV}: We have $M = \oplusM(M_1, i, M_2)$, and we define $\phi_M(M) = \oplusB(\phi_M(M_1), i, \phi_M(M_2))$.
\end{itemize}

\subsection{Decorated trees}

In \tambijref{}, the authors defined a family of labeled plane trees called \emph{decorated trees}, which is related to NSP-maps. Here, we take the convention that the root of a tree has depth $0$. For a vertex in a tree, its \tdef{direct sub-trees} are sub-trees induced by one of its children. The \tdef{traversal order} of leaves in a tree is the order induced by the prefix traversal of the tree. A \tdef{decorated tree} is a rooted plane tree where each leaf has a label at least $-1$, satisfying the following conditions.
\begin{enumerate}
\item The label of a leaf must be strictly less than the depth of its parent.
\item For each internal node with depth $p>0$, it has at least one descendant leaf with a label at most $p-2$.
\item Let $t$ be an internal node of depth $p$ and $T'$ one of its direct sub-trees. For any leaf $\ell$ in $T'$ whose label is the same as the depth $p$ of $t$, leaves in $T'$ coming before $\ell$ in traversal order have labels at least $p$.
\end{enumerate}

The right side of Figure~\ref{fig:bij-T} is an example of a decorated tree. We denote by $\mathcal{T}_n$ the set of decorated trees with $n$ edges (internal and external), and $\mathcal{T} = \bigcup_{n \geq 1} \mathcal{T}_n$ the set of all decorated trees. A leaf with label $-1$ is called a \tdef{free leaf}, and we denote by $\flT(T)$ the number of free leaves of a decorated tree $T$. By considering the first condition of decorated trees on nodes of depth $1$, we know that every decorated tree containing an internal node other than the root has at least one free leaf. This also applies to decorated trees without any internal node other than the root, where all leaves must be free leaves.

Much alike previous objects, decorated trees has a recursive decomposition in four cases exactly like that of $\beta$-(1,0) trees, illustrated in Figure~\ref{fig:tree-case}. We now define some operations on decorated trees that will be used in the recursive decomposition.

Let $T$ be a decorated tree and $k = \flT(T)$. For any $1 \leq i \leq k$, we denote by $\DeltaT(T,i)$ the tree obtained by attaching $T$ to a new vertex $v$ as a direct sub-tree, adding $1$ to each leaf except for the last $i$ free leaves. If $T$ is of type II, we define $\PiT(T)$ as the tree obtained by removing the root and subtracting $1$ from each leaf label, except free leaves. These functions are illustrated in the left part of Figure~\ref{fig:deco-fct}. We have the following properties.

\begin{prop} \label{prop:pi-delta-t}
Let $T$ be a decorated tree and $k=\flT(T)$. For any integer $i$ between $1$ and $k$, the tree $\DeltaT(T,i)$ is a decorated tree, and we have $\PiT(\DeltaT(T,i)) = T$ and $\flT(\DeltaT(T,i)) = i$. Moreover, for any decorated tree $T'$ of type II, $\PiT(T')$ is also a decorated tree, and if $\PiT(T') = T$, there exists a unique integer $j$ between $1$ and $k$ such that $T' = \DeltaT(T,j)$.
\end{prop}
\begin{proof}
We first prove that $\DeltaT(T,i) \in \mathcal{T}$ by checking all the conditions. By construction, comparing to $T$, the depth of every node and every leaf label are both incremented by $1$, except for some free leaves. Therefore, since $T$ is a decorated tree, the first condition is satisfied by all leaves since free leaves satisfy it automatically. For the second condition, we only need to check the unique child $u$ of the root. Since $u$ is of depth $1$, the second condition on $u$ is satisfied by the existence of $i \geq 1$ free leaves. For the third condition, the only possible violation is that a free leaf precedes a leaf with label $0$, which is impossible by construction. Therefore, $\DeltaT(T,i) \in \mathcal{T}$. From construction, we know that $\PiT(\DeltaT(T,i)) = T$ and $\flT(\DeltaT(T,i)) = i$.

We now prove that $\PiT(T') \in \mathcal{T}$. Comparing to $T'$, the depth of every node and the label of every leaf are both decremented by $1$, except for free leaves. Therefore, since $T'$ is a decorated tree, all conditions are automatically satisfied, thus $\PiT(T') \in \mathcal{T}$. For $\PiT(T')=T$, to go from $T$ to $T'$, the only ambiguity lies in the free leaves, since some free leaves of $T$ can get a label $0$ in $T'$. However, due to the third condition, all leaves with label $0$ in $T'$ must precede free leaves in traversal order, and there is at least one free leaf. Therefore, $T'$ must be of the form $\DeltaT(T,j)$.
\end{proof}

\begin{figure}
\centering
\includegraphics[page=14, width=\textwidth]{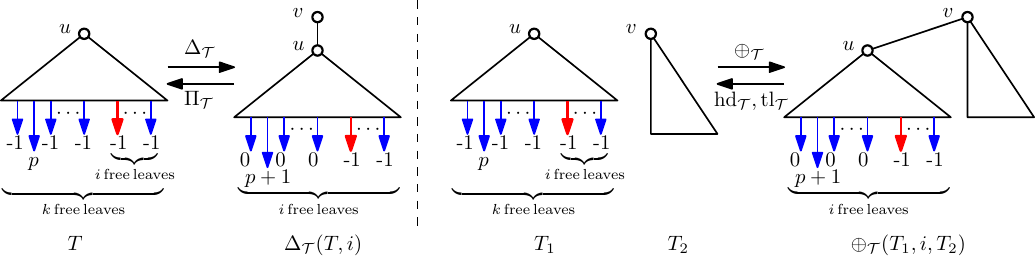}
\caption{Functions related to the recursive decomposition of decorated trees} \label{fig:deco-fct}
\end{figure}

Let $T_1, T_2$ be two decorated trees, and $1 \leq i \leq \flT(T_1)$. We denote by $\oplusT(T_1,i,T_2)$ the decorated tree obtained by identifying the roots of $\DeltaT(T_1, i)$ and $T_2$, from left to right. Conversely, given a decorated tree $T$ of type IV, we define $\hdT(T)$ and $\tlT(T)$ as follows: we first split the root to get two trees, one (denoted by $T_{hd}$) that contains the leftmost direct sub-tree, the other (denoted by $T_{tl}$) being the remainder, and we define $\hdT(T) = \PiT(T_{hd})$ and $\tlT(T)=T_{tl}$. We observe that both $T_{hd}$ and $T_{tl}$ satisfy all conditions of decorated trees, and by Proposition~\ref{prop:pi-delta-t}, we know that $\hdT(T)$ and $\tlT(T)$ are decorated trees. We also easily verify that $\tlT(T)$ is also a decorated tree. These functions are illustrated on the right side of Figure~\ref{fig:deco-fct}. We have the following properties.

\begin{prop} \label{prop:oplus-hd-tl-t}
For two decorated trees $T_1, T_2$ and an integer $i$ between $1$ and $\flT(T_1)$, the tree $T = \oplusT(T_1, i, T_2)$ is also a decorated tree, with $\hdT(T) = T_1$, $\tlT(T)=T_2$ and $\flT(T) = i + \flT(T_2)$.

Furthermore, for any decorated tree $T'$ of type IV, if $\hdT(T') = T_1$ and $\tlT(T') = T_2$, then there exists a unique integer $j$ between $1$ and $\flT(T_1)$ such that $T' = \oplusT(T_1,j,T_2)$.
\end{prop}
\begin{proof}
We first prove that $T \in \mathcal{T}$. We already know from Proposition~\ref{prop:pi-delta-t} that $\DeltaT(T_1,i)$ is a decorated tree. The only condition of decorated trees to check on $T$ is the third, but since it only concerns one direct sub-tree of an internal node, the fact that both $\DeltaT(T_1,i)$ and $T_2$ are decorated trees implies that the third condition is also satisfied by $T$, thus $T \in \mathcal{T}$. From the construction, we have $\hdT(T) = T_1$, $\tlT(T)=T_2$ and $\flT(T) = i + \flT(T_2)$.

For the second point, we can split $T'$ into $T'_{hd}$ and $T'_{tl}$ as in the definition of $\tlT$. We then have $T'_{tl} = T_2$ and $\PiT(T'_{hd}) = \hdT(T') = T_1$, and the result follows from Proposition~\ref{prop:pi-delta-t}.
\end{proof}

We now consider the recursive decomposition of decorated trees of each tree type.
\begin{itemize}
\item \textbf{Type I}: This is the base case, where $\flT(T)=1$.
\item \textbf{Type II}: The function $\PiT$ applies in this case. By Proposition~\ref{prop:pi-delta-t}, all decorated trees of type II can be obtained once and only once using $\DeltaT$.
\item \textbf{Type III}: Let $\ell$ be the leftmost child of $v$, which is a leaf. By the first condition of decorated trees, $\ell$ must be a free leaf. By removing $\ell$, we get a smaller decorated tree $T'$ with $\flT(T') = \flT(T)-1$, satisfying all conditions. Every decorated tree of type III can be obtained by attaching a free leaf on the left side of the root of a decorated tree.
\item \textbf{Type IV}: In this case, the functions $\hdT$ and $\tlT$ apply. By Proposition~\ref{prop:oplus-hd-tl-t}, all decorated trees of type IV can be obtained once and only once using $\oplusT$.
\end{itemize}

We observe that this recursive construction is isomorphic to that of $\beta$-(1,0) trees, with the statistic $\flT$ corresponding to $\rootB$. We thus define the canonical map $\phi_T$ from the set of decorated trees $\mathcal{T}$ to the set of $\beta$-(1,0) trees $\mathcal{B}$ by identifying all four cases of recursive decomposition. Figure~\ref{fig:decotree} gives an example of the image of a decorated tree by $\phi_T$.

\begin{figure}
\centering
\includegraphics[page=7, scale=0.7]{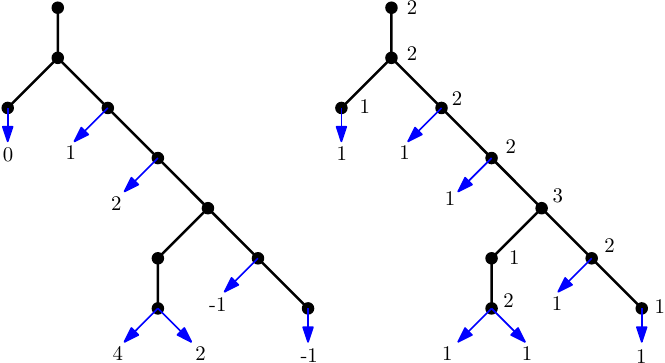}
\caption{A decorated tree $T$ and its corresponding $\beta$-(1,0) tree $\phi_T(T)$} \label{fig:decotree}
\end{figure}

\subsection{Synchronized intervals}

A \tdef{Dyck path} is a non-empty path on $\mathbb{Z}^2$ consisting of up steps $u=(1,1)$ and down steps $d=(-1,1)$ that starts from the origin, ends on the $x$-axis and always stays above the $x$-axis. We can also see a Dyck path as a word in the alphabet $\{ u,d \}$. It is clear that a Dyck path always consists of an even number of steps, and it has as many up steps as down steps. The \tdef{size} of a Dyck path is defined as the number of steps divided by $2$. A Dyck path of size $n$ thus has $n$ up steps and $n$ down steps. Let $D$ be a Dyck path of size $n$, and $i_1, i_2, \ldots, i_n$ be the indices in $D$ such that $\ell_{i_k}=u$ for all $k$. The type $\opname{Type}(D)$ of $D$ is a word $w$ in the alphabet $\{ N,E \}$ of length $n-1$, such that, for any $k$ from $1$ to $n-1$, $w_k=E$ if $i_{k+1}=i_k+1$, and $w_k=N$ otherwise. In other words, $w_k=E$ if and only if the $k^{\rm{th}}$ up step of $D$ is followed by an up step. 

The well-known Tamari lattice can be defined on Dyck paths. Let $\mathcal{D}_n$ be the set of Dyck paths of size $n$, and $\mathcal{D} = \cup_{n \geq 0} \mathcal{D}_n$ be the set of all Dyck paths (including the empty path). We now define a covering relation on $\mathcal{D}_n$. Let $D$ be an element of $\mathcal{D}_n$ that can be written (as a word) as $D = V d D_1 W$, where $V,W$ are words in letters $u,d$ and $D_1$ is a strictly smaller Dyck path. We then construct $D' = V D_1 d W$, which is also a Dyck path, and we say that $D$ covers $D'$. The \tdef{Tamari lattice} $(\mathcal{D}_n, \preceq)$ of order $n$ is given by the transitive closure $\preceq$ of the covering relation we just defined. Readers are referred to \tamref for a more detailed description of the Tamari lattice.

An \tdef{interval} in the Tamari lattice $(\mathcal{D}_n, \preceq)$ of order $n \geq 1$ is simply a pair of comparable Dyck paths $[D_1, D_2]$ with $D_1 \preceq D_2$. We observe here that the empty Dyck path is not allowed in any interval. An interval $(D_1, D_2)$ is called \tdef{synchronized} if $D_1$ and $D_2$ are of the same type. We denote by $\mathcal{I}_n$ the set of synchronized intervals in the Tamari lattice of order $n$, and $\mathcal{I} = \cup_{n\geq1}\mathcal{I}_n$ the set of all synchronized intervals. Synchronized intervals are in bijection with intervals in the so-called ``generalized Tamari lattices'' (\textit{cf.} \tamref). In \tambijref, a bijection between synchronized intervals and decorated trees was established, which motivates our study of synchronized intervals here. 

We now investigate a recursive decomposition of synchronized intervals. A \tdef{contact} of a Dyck path $D$ is a lattice point of $D$ that is also on the $x$-axis. For example, the Dyck path $uduududduudd$ has $4$ contacts. For a synchronized interval $I = [P,Q]$, we denote by $\contI(I)$ the number of contacts of the smaller path $P$ \emph{minus 1}. We can also say that the initial contact $(0,0)$ of $P$ is ignored in $\contI(D)$. A \tdef{properly pointed synchronized interval} is a synchronized interval $[P,Q]$ with a distinguished non-initial contact, which is also written as $[P^\ell P^r,Q]$, where $P=P^\ell P^r$ is split by the distinguished contact into two sub-paths $P^\ell$ and $P^r$, where $P^\ell$ is a non-empty Dyck path. We denote by $\mathcal{I}^\bullet_n$ the set of properly pointed synchronized intervals of size $n$. It is clear that the number of properly pointed synchronized intervals corresponding to $I$ is $\contI(I)$.

We now reformulate a recursive decomposition proved in \tambijref.

\begin{prop} \label{prop:int-decomp}
Let $[P,Q]$ be a synchronized interval in the Tamari lattice of order $n$. There is a unique way to decompose the two Dyck paths $P$ and $Q$ as
\[ P = u P_1^\ell d P_1^r P_2, \,\,\, Q = u Q_1 d Q_2, \]
where the sub-paths $P_1^\ell, P_1^r, P_2, Q_1, Q_2$ satisfies:
\begin{itemize}
\item Each path is either empty or a Dyck path;
\item $P_1^\ell P_1^r$ is empty if and only if $Q_1$ is empty, $P_2$ is empty if and only if $Q_2$ is empty;
\item When not empty, $[P_1^\ell P_1^r, Q_1]$ is a properly pointed synchronized interval, and $[P_2,Q_2]$ is synchronized interval;
\item If $P_1^\ell$ is empty, then both $P_1^r$ and $Q_1$ are empty.
\end{itemize}
Let $\epsilon$ be the empty path. The decomposition above is a bijection between $\mathcal{I}$ and $\left(\{[\epsilon, \epsilon]\} \cup \bigcup_{n\geq1} \mathcal{I}^\bullet_n \right) \times \left( \{ [\epsilon,\epsilon]\} \cup \mathcal{I} \right)$.
\end{prop}

\begin{figure}
  \centering
  \includegraphics[page=18,scale=0.8]{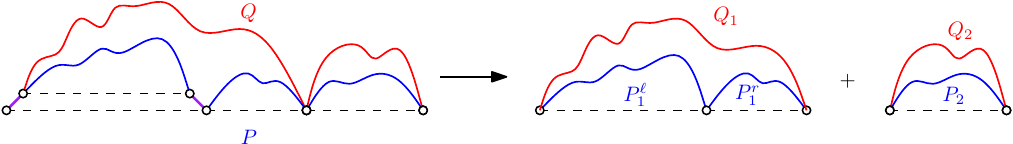}
  \caption{Decomposition of synchronized intervals}
  \label{fig:int-decomp}
\end{figure}

Figure~\ref{fig:int-decomp} illustrate this decomposition. A proof of Proposition~\ref{prop:int-decomp} can be found in \tambijref{}. Let $I=[P,Q]$ be a synchronized interval, and let $P=u P_1^\ell d P_1^r P_2$ and $Q=u Q_1 d Q_2$ be its decomposition as in Proposition~\ref{prop:int-decomp}. According to whether $P_1^\ell$ and $P_2$ are empty or not, the interval $I$ can be put into the following four types (see Figure~\ref{fig:int-type}):

\begin{itemize}
\item \textbf{Type I}: Both $P_1^\ell$ and $P_2$ are empty, which implies $P = Q = ud$;
\item \textbf{Type II}: $P_1^\ell$ is not empty, but $P_2$ is empty, which implies $P=u P_1^\ell d P_1^r$ and $Q=u Q_1 d$, where $[P_1^\ell P_1^r,Q_1]$ is a synchronized interval;
\item \textbf{Type III}: $P_1^\ell$ is empty, but $P_2$ is not empty, which implies $P=udP_2$ and $Q=udQ_2$, where $[P_2,Q_2]$ is a synchronized interval;
\item \textbf{Type IV}: Neither $P_1^\ell$ nor $P_2$ is empty, which implies that both $[P_1^\ell P_1^r,Q_1]$ and $[P_2,Q_2]$ are synchronized intervals.
\end{itemize}

\begin{figure}
  \centering
  \includegraphics[page=19,scale=0.8]{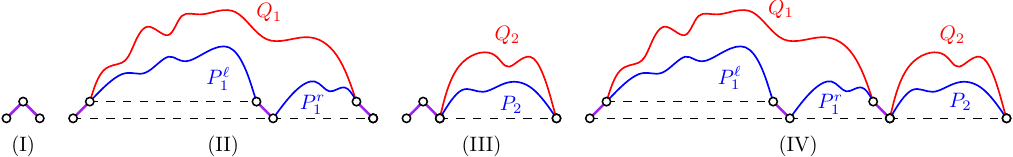}
  \caption{Synchronized intervals of different types}
  \label{fig:int-type}
\end{figure}

We now define a few operations according to the types above. Let $I=[P,Q]$ be a synchronized interval, and $k=\contI(I)$. For any $1\leq i\leq k$, we define $\DeltaI(I,i)$ to be the synchronized interval $[u P^\ell d P^r, u Q d]$, where $[P^\ell P^r, Q]$ is the properly pointed synchronized interval obtained by distinguishing the $i^{\rm th}$ contact \emph{from right to left}. If $I$ is of type II, it can be written as $I = [u P_1^\ell d P_1^r, u Q_1 d]$, and in this case we define a synchronized interval $\PiI(I) = [P_1,Q_1]$, where $P_1 = P_1^\ell P_1^r$. Let $I_1, I_2$ be two synchronized intervals, and $1 \leq i \leq \contI(I_1)$. We denote by $\oplusI(I_1, i, I_2)$ the synchronized interval obtained by concatenating $\DeltaI(I_1, i)$ with $I_2$. Conversely, given a synchronized interval $I=[uP_1^\ell dP_1^r P_2, uQ_1dQ_2]$ of type IV, we define $\hdI(I)=[P_1^\ell P_1^r, Q_1]$ and $\tlI(I)=[P_2,Q_2]$. The following properties are corollaries of Proposition~\ref{prop:int-decomp}.

\begin{prop} \label{prop:pi-delta-i}
  Let $I$ be a synchronized interval and $k=\contI(I)$. For any integer $i$ between $1$ and $k$, we have $\PiI(\DeltaI(I,i))=I$ and $\contI(\DeltaI(I,i))=i$. Moreover, every synchronized interval $I'$ of type II can be uniquely written as $I'=\DeltaI(J,i)$, where $J$ is a synchronized interval.
\end{prop}

\begin{prop} \label{prop:oplus-hd-tl-i}
  For two synchronized intervals $I_1, I_2$ and an integer $1 \leq i \leq \contI(I_1)$, let $I=\oplusI(I_1,i,I_2)$. We have $\hdI(I)=I_1$, $\tlI(I)=I_2$ and $\contI(I)=i+\contI(I_2)$. Furthermore, every synchronized interval $I'$ of type IV can be uniquely written as $I' = \oplusI(I_1',j,I_2')$ with $I_1', I_2' \in \mathcal{I}$.
\end{prop}

We observe that the recursive construction of synchronized intervals is also isomorphic to that of $\beta$-(1,0) trees, with the statistic of lower path contacts $\contI(I)$ that plays the same role as the statistic of root label $\rootB(B)$. We thus define the canonical map $\phi_{I}$ from the set of synchronized intervals $\mathcal{I}$ to the set of $\beta$-(1,0) trees $\mathcal{B}$ by identifying all the cases of both recursive decompositions.

\section{Natural canonical bijections} \label{sec:nat-bij}

In this section, we discuss two bijections from \tambijref{} relating objects we mentioned above: the bijection $\mathrm{T}$ from NSP-maps to decorated trees, and the bijection $\mathrm{I}$ from decorated trees to synchronized intervals. We then show that both $\mathrm{T}$ and $\mathrm{I}$ are canonical bijections with respect to the recursive decompositions for each class of objects, albeit the fact that they are defined as direct bijections.

In \tambijref{}, the authors gave a bijection $\mathrm{T}$ between NSP-maps and decorated trees via the following depth-first search procedure. For an NSP-map $M$ with its root $e = (v \to u)$, we first delete $e$ in $M$, mark $u$ and $v$ as visited, then proceed a depth-first search starting from $u$ and the edge next to $e$ on the outer face, and at each time we explore a new vertex, the edges of the newly visited vertex should be then visited in clockwise order. Each newly visited edge leading to an already visited vertex of depth $p$ produces a leaf with label $p$, with the convention that $v$ is of depth $-1$. We thus have a tree with labels on its leaves, and it has been proved in \tambijref{} that it is always a decorated tree. This bijection is illustrated in Figure~\ref{fig:bij-T}.

\begin{figure}
\centering
\includegraphics[page=2, scale=0.75]{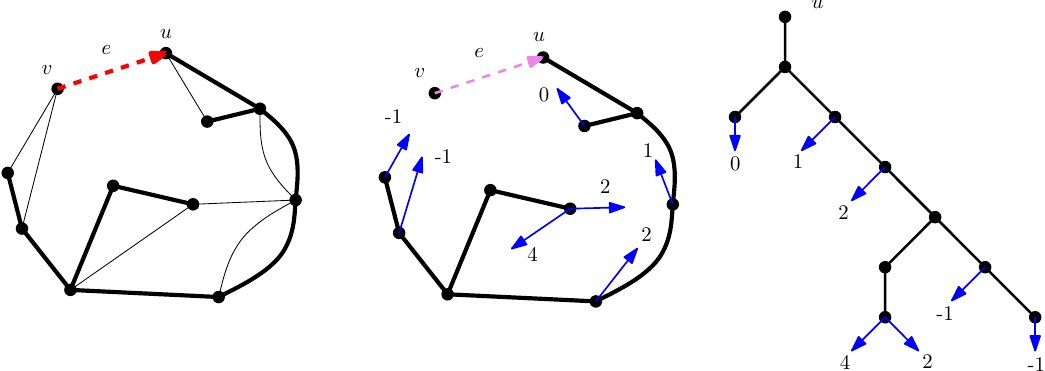}
\caption{The bijection $\mathrm{T}$ applied to an NSP-map $M$} \label{fig:bij-T}
\end{figure}

\begin{prop} \label{prop:T-bij}
For any $n \geq 1$, the function $\mathrm{T}$ is a bijection from $\mathcal{M}_n$ to $\mathcal{T}_n$ that transfers the $\opname{deg}$ statistic in $\mathcal{M}_n$ to $\flT$ in $\mathcal{T}_n$.
\end{prop}

A direct proof that does not depend on recursive decompositions can be found in \tambijref{}. In the following, we will also give a proof that $\mathrm{T}$ is a bijection using a recursive decomposition of decorated trees. More precisely, we prove that $\mathrm{T}$ illustrated in Figure~\ref{fig:bij-T} is the canonical bijection from our parallel decomposition of NSP-maps to decorated trees.

\begin{thm} \label{thm:T-recursive-canonical}
We have $\phi_M = \phi_T \circ \mathrm{T}$, where $\phi_M$ (\emph{resp.} $\phi_T$) is the canonical bijections from NSP-maps (\emph{resp.} decorated trees) to $\beta$-(1,0) trees.
\end{thm}
\begin{proof}
We proceed by induction on the size of NSP-maps. Let $M \in \mathcal{M}_n$, we verify that $\phi_M(M) = \phi_T(\mathrm{T}(M))$ according to the type of $M$ while supposing that any map $M'$ with a size smaller than $M$ satisfies $\phi_M(M') = \phi_T(\mathrm{T}(M'))$.
\begin{itemize}
\item \textbf{Type I}: This is the base case, and we easily observe that $\phi_M(M) = \phi_T(\mathrm{T}(M))$.
\item \textbf{Type II}: Figure~\ref{fig:T-canonical-case-2} gives an illustration of this case. We know that $M = \DeltaM(M', i)$ for a certain NSP-map $M'$ and $1 \leq i \leq \opname{deg}(M')$ by Proposition~\ref{prop:pi-delta-m}. By induction hypothesis, we only need to prove that $\mathrm{T}(M) = \DeltaT(\mathrm{T}(M'),i)$. Let $v$ be the root vertex of $M'$ and $(v,w)$ the root of $M'$. Consider the exploration process that constructs $\mathrm{T}(M)$. The two maps $M$ and $M'$ only differs in that $v$ in $M'$ is split into $v$ and $u$ in $M$, which does not affect the exploration process since they are all marked as visited at the beginning. Therefore, the exploration of $M$ starts with the edge $(u,w)$, and the rest is the same as that of $M'$. Ignoring the edge $(u,w)$, the two decorated trees diverge only in that some free leaves in $\mathrm{T}(M')$ become leaves labeled by $0$ in $\mathrm{T}(M)$ pointing to $u$. Therefore, $\PiT(T') = \mathrm{T}(M')$. Since $\flT(\mathrm{T}(M)) = i$ and $\PiT(T')=\mathrm{T}(M')$, by Proposition~\ref{prop:pi-delta-t}, we have $\mathrm{T}(M) = \DeltaT(\mathrm{T}(M'),i)$.
\item \textbf{Type III}: Let $e_1$ be the edge corresponding to the loop after contraction of the root, and $M'$ the map after deleting $e_1$. In the construction of $\mathrm{T}(M)$, the edge $e_1$ is visited last and is a free leaf, while the remainder is $\mathrm{T}(M')$.
\item \textbf{Type IV}: In this case, we know that $M = \oplusM(M_1, i, M_2)$ for two smaller NSP-maps $M_1, M_2$ and $1 \leq i \leq \opname{deg}(M_1)$. In the construction of $\mathrm{T}(M)$, the part of $M_1$ is visited last and mapped to the leftmost direct sub-tree of the root. Let $T_1 = \mathrm{T}(M_1)$ and $T_2 = \mathrm{T}(M_2)$, by induction hypothesis, similar to the type II case, we have $\mathrm{T}(\DeltaM(M_1,i)) = \DeltaT(T_1,i)$, thus $\mathrm{T}(M) = \oplusT(T_1, i, T_2) = \phi_M(M)$ by the definition of $\oplusT$.
\end{itemize}

We thus complete the induction and we have $\phi_M = \phi_T \circ \mathrm{T}$.
\end{proof}

\begin{figure}
\centering
\includegraphics[page=13, scale=1]{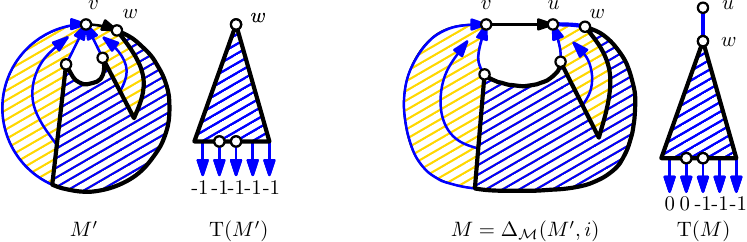}
\caption{Decorated trees in the recursive decomposition of NSP-maps of type II} \label{fig:T-canonical-case-2}
\end{figure}

Since both $\phi_M$ and $\phi_T$ are canonical recursive bijections, from Theorem~\ref{thm:T-recursive-canonical} we conclude that, albeit the non-recursive definition, $T$ is also the canonical recursive bijection from the set $\mathcal{M}$ of NSP-maps in parallel decomposition to the set $\mathcal{T}$ of decorated trees. We thus have Proposition~\ref{prop:T-bij} as a consequence of Theorem~\ref{thm:T-recursive-canonical}. Furthermore, we have the following observation.

\begin{prop}
Let $T$ be a decorated tree. Disregarding labels, $T$ and $\phi_T(T)$ are the same tree.
\end{prop}

Therefore, for a decorated tree $T$, the only difference between $T$ and $\phi_T(T)$ is that we have labels only on leaves of $T$, but on every node of $\phi_T(T)$. We want a more direct, non-recursive way to construct $\phi_T(T)$ from $T$ by converting between these two kinds of labels. We define the following map $\varphi_T$ from $\mathcal{T}$ to $\mathcal{B}$. Let $T$ be a decorated tree. For an internal node of $T$ of depth $p > 0$, we label it by the number of descendant leaves with labels at most $p-2$. We label the root of $T$ by $\flT(T)$. By changing all leaf labels to $1$, we get a tree with labels on every node that we denote by $\varphi_T(T)$. An example can be found in Figure~\ref{fig:decotree}.

\begin{prop}
For any decorated tree $T$, the tree $\varphi_T(T)$ is the $\beta$-(1,0) tree $\phi_T(T)$.
\end{prop}
\begin{proof}
We proceed by induction on $T$. In the following, we will consider $T$ in different types, and we suppose that any decorated tree $T'$ whose size is smaller than $T$ satisfies $\varphi_T(T') = \phi_T(T')$.
\begin{itemize}
\item \textbf{Type I}: This is the base case, and we verify that $\varphi_T(T) = \phi_T(T)$ in this case.
\item \textbf{Type II}: In this case, we have $T = \DeltaT(T',i)$ for some $T'$ and $i$, thus $\phi_T(T) = \DeltaB(\varphi_T(T'),i)$ by induction hypothesis, and we only need to prove that $\varphi_T(T) = \DeltaB(\varphi_T(T'),i)$. Let $u$ be the root of $T'$. The operation $\DeltaT$ increments leaf labels and depth of internal nodes by $1$, except for some free leaves. Therefore, $\varphi_T(T)$ is $\varphi_T(T')$ attached to a new node as root, only changing the label of $u$, which is $i$ in $\varphi_T(T)$ since there are $i$ free leaves in the descendants of $u$ in $T$. This construction agrees with $\DeltaB$.
\item \textbf{Type III}: In this case, $T$ is obtained by attaching a leaf $\ell$ to another tree $T'$. By the first condition of decorated trees, $\ell$ must be free, which implies that $\rootB(\varphi_T(T)) = \rootB(\varphi_T(T'))+1=\rootB(\phi_T(T))$. Other parts of $\varphi_T(T)$ also agree with those of $\phi_T(T)$.
\item \textbf{Type IV}: In this case, we have $T = \oplusT(T_1, i, T_2)$ for some $T_1$, $T_2$ and $i$, thus $\phi_T(T) = \oplusB(\varphi_T(T_1), i, \varphi_T(T_2))$ by induction hypothesis, and we only need to prove that $\varphi_T(T) = \oplusB(\varphi_T(T_1), i, \varphi_T(T_2))$. This is correct, since by the case of type II, we have $\varphi_T(\DeltaT(T_1,i)) = \DeltaB(\phi_T(T), i)$, and the identification of roots commutes with $\varphi_T$.
\end{itemize}
We have completed all cases, and by induction, we have $\varphi_T(T) = \phi_T(T)$ for all $T \in \mathcal{T}$.
\end{proof}

We now present the bijection $\mathrm{I}$ from decorated trees to synchronized intervals that was given in \tambijref{}. For a decorated tree $T$, we say that a leaf $\ell$ is the \tdef{certificate} of a non-root internal node $u$ if $\ell$ is the first leaf in left-to-right order that makes $u$ satisfy the second condition of decorated trees. We now define the \tdef{certificate-counting function} $c$ of $T$ as a function on leaves of $T$, such that for every leaf $\ell$, we have $c(\ell)$ the number of internal nodes whose certificate is $\ell$. This function $c$ is called the ``charges'' in \tambijref{}. We now define the bijection $\mathrm{I}$ from decorated trees to synchronized intervals as follows, with an example illustrated in Figure~\ref{fig:sync-bij-ex}. Given a decorated tree $T$, we perform a counter-clockwise depth-first traversal of $T$, and we get a Dyck path $\mathrm{Q}(T)$ that records the variation of depth in our traversal. For the other Dyck path $\mathrm{P}(T)$, we perform the same traversal starting with $\mathrm{P}(T)$ an empty word, then at each time we explore a new edge, we append an up step $u$ to $\mathrm{P}(T)$, and when we meet a leaf $\ell$, we append $d^{c(\ell)+1}$ to $\mathrm{P}(T)$. We then have the full $\mathrm{P}(T)$ after the traversal. The two paths $\mathrm{P}(T)$ and $\mathrm{Q}(T)$ form a synchronized interval, denoted by $\mathrm{I}(T)=[\mathrm{P}(T), \mathrm{Q}(T)]$.

\begin{figure}
  \begin{center}
    \includegraphics[page=20,scale=0.85]{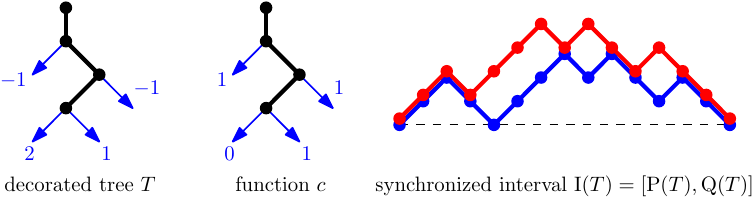}
  \end{center}
  \caption{An example of the bijection $\mathrm{I}$ from decorated trees to synchronized intervals}
  \label{fig:sync-bij-ex}
\end{figure}

\begin{prop}[Theorem~4.11 in \tambijref{}] \label{prop:sync-bij}
  The function $\mathrm{I}$ is a bijection from $\mathcal{I}$ to $\mathcal{T}$.
\end{prop}

Readers are referred to \tambijref{} for a description of the reverse direction and a detailed proof of Proposition~\ref{prop:sync-bij}.

We have the following lemma that relates free leaves in a decorated tree to contacts of the lower path of the corresponding synchronized interval.

\begin{lem} \label{lem:free-leaf-contact}
  Let $T$ be a decorated tree and $I = \mathrm{I}(T) = [P,Q]$ its corresponding synchronized interval. For a free leaf $\ell$ of $T$, let $P_\ell$ be the initial segment of $P$ given by the traversal just after visiting $\ell$. Then $P_\ell$ is a Dyck path.
\end{lem}
\begin{proof}
  Let $\ell_1, \ldots, \ell_k$ be the leaves that comes before $\ell$ in traversal order, and $m$ the number of non-root internal nodes visited before visiting $\ell$. By the construction of $\mathrm{I}(T)$, we have $m+k+1$ up steps and $c(\ell) + \sum_{i=1}^k c(\ell_i) + k + 1$ down steps in $P_\ell$. Let $u$ be a non-root internal node that is visited before $\ell$. If $u$ is not an ancestor of $\ell$, then its certificate, which is one of its descendant leaf, must precede $\ell$ in the traversal order. If $u$ is an ancestor of $\ell$, since the free leaf $\ell$ has a label $-1$, the certificate of $u$ cannot come after $\ell$. Therefore, all non-root internal nodes have their certificate being $\ell$ or preceding $\ell$, which make $c(\ell)+\sum_{i=1}^k c(\ell_i) \geq m$. However, since $P$ is a Dyck path, we must have equality, which means $P_\ell$ is also a Dyck path.
\end{proof}

By Lemma~\ref{lem:free-leaf-contact}, we see that free leaves in a decorated tree are in one-to-one correspondence with non-initial contacts in the lower path of the corresponding synchronized interval.

We now prove that the bijection $\mathrm{I}$ is canonical with respect to recursive decompositions of decorated trees and synchronized intervals.

\begin{thm} \label{I-recursive-canonical}
  We have $\phi_T = \phi_I \circ \mathrm{I}$, where $\phi_T$ (\emph{resp.} $\phi_I$) is the canonical bijection from decorated trees (\emph{resp.} synchronized intervals) to $\beta$-(1,0) trees.
\end{thm}
\begin{proof}
  We proceed by induction on the size of decorated trees. Let $T$ be a decorated tree. Suppose that every decorated tree $T'$ with strictly fewer nodes than $T$ satisfies $\phi_T(T') = \phi_I(\mathrm{I}(T'))$. We now prove that $\phi_T(T) = \phi_I(\mathrm{I}(T))$. We denote by $I = [P,Q] = \mathrm{I}(T)$ the corresponding synchronized interval of $T$.
  \begin{itemize}
  \item \textbf{Type I}: This is the base case, and we easily check $\phi_T(T)=\phi_I(\mathrm{I}(T))$.
  \item \textbf{Type II}: We have $T = \DeltaT(T',i)$ with $T'$ smaller than $T$ in this case. Let $[P',Q']=\mathrm{I}(T')$. By the definition of $\mathrm{I}$, we have $Q=u Q' d$. By the definition of $\DeltaT$, the certificate of a non-root node in $T'$ remains the same in $T$. For the has-been root of $T'$, since its depth is $1$ in $T$, its certificate is the first free leaf in $T$. Therefore, by Lemma~\ref{lem:free-leaf-contact}, we have $P = u P_1' d P_2'$ with $P_1'$ a non-empty Dyck path, and $\contI(P_2')=i$. We thus have $I = \DeltaI(\mathrm{I}(T'),i)$, and by the induction hypothesis we conclude this case.
  \item \textbf{Type III}: In this case, the leftmost child $\ell$ of the root of $T$ is a leaf, and by the definition of decorated trees, it has a label $-1$. Let $T'$ be the tree $T$ with $\ell$ removed, which remains a decorated tree. Let $[P',Q']=\mathrm{I}(T')$, it is clear that $Q = udQ'$. Since $\ell$ has no influence on the certificate of all the non-root internal nodes, we have $P=udP'$. Therefore, $I$ is also of type III and constructed from $\mathrm{I}(T')$.
  \item \textbf{Type IV}: We have $T = \oplusT(T_1, i, T_2)$ with $T_1, T_2$ smaller than $T$ in this case. Let $[P_1, Q_1] = \mathrm{I}(T_1)$ and $[P_2, Q_2]=\mathrm{I}(T_2)$. From the definition of $\oplusT$, we have $Q = u Q_1 d Q_2$. For the lower path $P$, we observe that the certificate of a node $u$ in $T_2$ remains the same in $T$, as the certificate must be a descendant leaf of $u$. Therefore, the part of $P$ due to traversal in the part due to $T_2$ is exactly $P_2$. For the part due to $T_1$, we can refer to the type II case. Therefore, we have $P = u P_{1,1} d P_{1,2} P_2$, with $\contI(P_{1,2})=i$, which then implies that $I = \oplusI(\mathrm{I}(T_1),i,\mathrm{I}(T_2))$. By the induction hypothesis, we conclude this case.
  \end{itemize}

  We have verified all four cases, and by induction on the size of $T$, we have that $\phi_T = \phi_I \circ \mathrm{I}$.
\end{proof}

\section{Map duality, involution $\boldsymbol{h}$ and order-reversing involution in the Tamari lattice} \label{sec:inv}

In this section, we study the mysterious involution $h$ on $\beta$-(1,0) trees defined by Claesson, Kitaev and Steingr\'imsson in \dnsref{}. The following description of $h$ is adapted from \dsymref{} to our recursive decomposition.

The involution $h$ is defined recursively on the set of $\beta$-(1,0) trees as follows. We first define the \tdef{rightmost path} of a plane tree $T$ to be the path from the root to the rightmost leaf, and $\rpath(T)$ the number of edges on the rightmost path. Let $B$ be a $\beta$-(1,0) tree. We will define $h(B)$ according to the tree type of $B$, as illustrated in Figure~\ref{fig:h-desc-def}.
\begin{itemize}
\item \textbf{Type I}: This is the base case, which is the tree with only one leaf. In this case, $h(B) = B$.
\item \textbf{Type II}: $B = \DeltaB(B',i)$ for some $B'$ and $i$. To construct $h(B)$, we first take the $\beta$-(1,0) tree $h(B')$, then add $1$ to the first $i$ nodes on its rightmost path, starting from the root, and we finish by attaching a new leaf as the rightmost child of the $i$-th node.
\item \textbf{Type III}: In this case, let $B'$ be the $\beta$-(1,0) tree resulted from deleting the leftmost child of the root and decreasing the root label by $1$. To get $h(B)$ from $h(B')$, we attach a new leaf to the rightmost leaf of $h(B')$, which becomes an internal node.
\item \textbf{Type IV}: In this case $B = \oplusB(B_1, i, B_2)$ for some $B_1, B_2$ and $i$. Let $B_1' = h(\DeltaB(B_1,i))$, $B_2' = h(B_2)$ and $\ell$ the rightmost leaf of $B_2$. By replacing $\ell$ with the root of $B_1'$ labeled by $1$ in $B_2'$, we get the $\beta$-(1,0) tree that we define as $h(B)$.
\end{itemize}

\begin{figure}
\centering
\includegraphics[page=15, width=\textwidth]{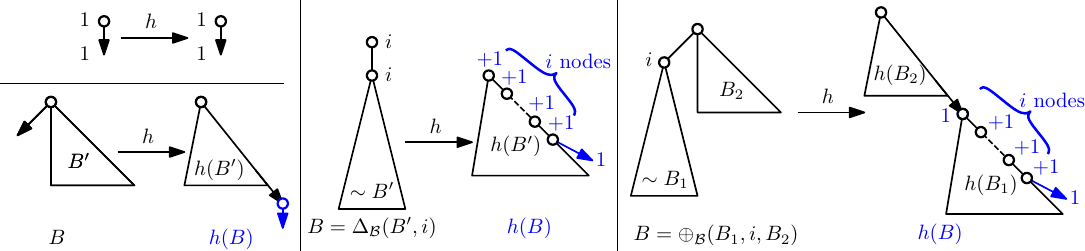}
\caption{Definition of $h$} \label{fig:h-desc-def}
\end{figure}

Readers familiar with the definition of the involution $h$ in \dsymref{} may notice that our definition here differs from the original in the ``decomposable cases'' (type III, IV), where we decompose from left to right while the original definition proceeds from right to left. However, we can see that the two definitions are the same by looking at the result of successive iterations as in Figure~5 of \dinvref{}.

The following results were proved in \dsymref{} and \dinvref{}.

\begin{thm}[\dsymref{}] \label{thm:h-involution}
The function $h$ is an involution on $\mathcal{B}_n$ for every $n \geq 1$. Furthermore, for a $\beta$-(1,0) tree $B$, we have $\rootB(B) = \rpath(h(B))$.
\end{thm}

\begin{thm}[\dinvref{}] \label{thm:h-autodual}
The number of fixed points of $h$ in $\mathcal{B}_n$ is exactly the number of self-dual NSP-maps in $\mathcal{M}_n$.
\end{thm}

In the following, we will give a bijective proof of both theorems. In fact, we prove that the function $h$ on $\beta$-(1,0) trees corresponds via $\phi_M$ to the duality on NSP-maps. We start by describing the counterpart of $h$ in decorated trees, denoted by $\hT$. Let $T$ be a decorated tree. We define recursively $\hT(T)$ according to the tree type of $T$, as illustrated in Figure~\ref{fig:h-deco-def}.
\begin{itemize}
\item \textbf{Type I}: This is the base case, the tree with only one leaf, which is free. In this case, $\hT(T) = T$.
\item \textbf{Type II}: $T = \DeltaT(T',i)$ for some $T'$ and $i$. To construct $\hT(T)$, we first take $\hT(T')$, then attach a free leaf to the $i$-th node on the rightmost path as the rightmost child.
\item \textbf{Type III}: Let $T'$ be $T$ after removing $\ell$. To construct $\hT(T)$, we take $\hT(T')$ and attach a free leaf to its rightmost leaf, which becomes an internal node.
\item \textbf{Type IV}: $T = \DeltaT(T_1, i, T_2)$ for some $T_1, T_2$ and $i$. Let $T_1' = \hT(\DeltaT(T_1,i))$, $T_2' = \hT(T_2)$, and $\ell$ be the rightmost leaf of $T_2'$. By adding $\rpath(T_2')$ to every leaf of $T_1'$ except the rightmost one, then replacing $\ell$ by the resulting tree in $T_2'$, we get $\hT(T)$.
\end{itemize}

\begin{figure}
\centering
\includegraphics[page=16, width=\textwidth]{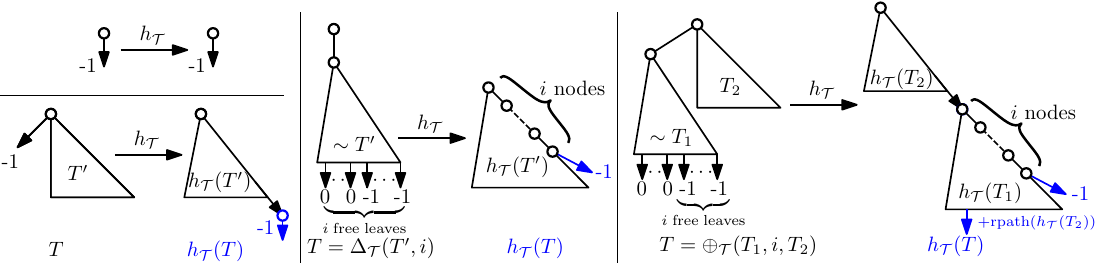}
\caption{Definition of $\hT$} \label{fig:h-deco-def}
\end{figure}

The following proposition confirms that $\hT$ is the equivalent of $h$ for decorated trees.
\begin{prop} \label{prop:hT-equiv-h}
  The function $\hT$ is equivalent to $h$ on $\beta$-(1,0) trees under the conjugation of the canonical bijection $\phi_T$ of decorated trees, that is, we have $\hT = \phi_T^{-1} \circ h \circ \phi_T$.
\end{prop}
\begin{proof}
Since $\phi_T$ is a bijection, we will prove instead $h \circ \phi_T = \phi_T \circ \hT$ by induction on the size of $\beta$-(1,0) trees. In the following, we will use the description of $\phi_T$ by $\varphi_T$. The first three cases can be easily checked, and we concentrate on type IV case where $T = \DeltaT(T_1, i, T_2)$. In $\phi_T(\hT(T))$, the part coming from $T_2$ is identical to that in $h(\phi_T(T_2))$ by induction hypothesis. For the part coming from $T_1$, in both $\phi_T(\hT(T))$ and $h(\phi_T(T))$, we have added $\rpath(\hT(T_2))$ to each leaf (except the rightmost one, which is always free) while the depth of each node in this part also increases by $\rpath(\hT(T_2))$. We thus conclude the induction.
\end{proof}

\begin{figure}
\centering
\includegraphics[page=6, scale=0.8]{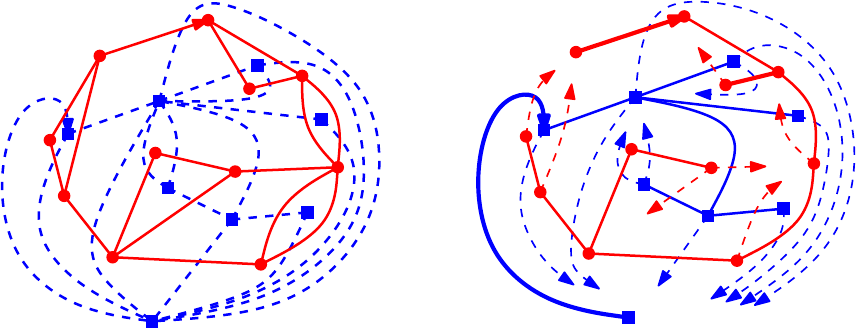}
\caption{The decorated tree of an NSP-map and its dual} \label{fig:deco-dual}
\end{figure}

We will now see that the duality on NSP-maps is exactly the involution $\hT$ on corresponding decorated trees. Figure~\ref{fig:deco-dual} gives an example of how map duality affects the decorated tree of an NSP-map.
\begin{thm} \label{thm:duality-equiv-hT}
  The decorated trees of an NSP-map and its dual are related by the involution $\hT$, or equivalently, we have $\dual = \mathrm{T}^{-1} \circ \hT \circ \mathrm{T}$.
\end{thm}
\begin{proof}
We only need to prove that $\mathrm{T} \circ \dual = \hT \circ \mathrm{T}$. We will proceed by induction on the size NSP-maps. Let $M$ be an NSP-map, $e$ its root and $v$ its root vertex. We suppose that any NSP-map $N$ with fewer edges than $M$ satisfies $\mathrm{T}(\dual(N)) = \hT(\mathrm{T}(N))$. By the construction of $\dual(M)$, the rightmost path of $\mathrm{T}(\dual(M))$ contains dual edges of edges adjacent to $v$ in $M$, except the edge that precedes $e$ on the outer face. With this observation, we analyze all four possible types of $M$.
\begin{itemize}
\item \textbf{Type I}: It is clear that $\mathrm{T}(\dual(M)) = \hT(\mathrm{T}(M))$ in this base case.
\item \textbf{Type II}: This case is illustrated by the left part of Figure~\ref{fig:deco-dual-case}. In this case, $M = \DeltaM(M',i)$ for some $M'$ and $i$. We can get $\dual(M)$ from $\dual(M')$ by adding the extra dual edge of the new edge in $M$, which is the root and is adjacent to the outer face, thus also a free leaf in $\mathrm{T}(\dual(M))$. This extra free leave does not alter the DFS process of $\dual(M)$, since it visits the root vertex of $\dual(M)$. This new free leaf is the rightmost leaf of $\mathrm{T}(\dual(M))$, and $\opname{deg}(M) = i$, implying that it must be attached to the $i$-th vertex on the rightmost path of $\mathrm{T}(\dual(M'))$, which agrees with the definition of $\hT$ in this case.
\item \textbf{Type III}: In this case, the core face of $M$ is adjacent to the root $e$ and another edge $e'$. By deleting $e'$, we have another NSP-map $M'$, and $\dual(M)$ can be obtained by inserting a vertex in the middle of the dual edge of $e$ in $\dual(M')$. Since $e$ corresponds to the rightmost leaf in $\mathrm{T}(\dual(M'))$, we can get $\mathrm{T}(\dual(M))$ by attaching a free leaf to the rightmost leaf of $\mathrm{T}(\dual(M'))$, which agrees with the definition of $\hT$ in this case.
\item \textbf{Type IV}: This case is illustrated by the right part of Figure~\ref{fig:deco-dual-case}. In this case, $M = \oplusM(M_1, i, M_2)$ for some $M_1, M_2$ and $i$. We can see $M$ as $M_2$ containing $\DeltaM(M_1,i)$ in the core face with the roots identified. Let $v$ be the root vertex of $\dual(M)$ and $g$ the dual vertex of the face $f$ that separates $\DeltaM(M_1,i)$ from $M_2$ in $M$. In the DFS process of $M_2$, the dual vertex $g$ is on the rightmost path, and the rightmost child of $g$ is a free leaf corresponding to the root of $M_2$. We consider the DFS process of $M$, which is identical to that of $M_2$ until we reach $g$. The first descendant (the rightmost child in the DFS tree) of $g$ is now a dual vertex $h$ in the part of $\DeltaM(M_1,i)$. Since the dual of the part $\DeltaM(M_1,i)$ is only connected to $g$ and the dual root vertex, $g$ plays the role of the dual vertex of the outer face of $\DeltaM(M_1,i)$, and the exploration of that part results in a direct sub-tree that is a modified version of $\mathrm{T}(\dual(\DeltaM(M_1,i)))$ where all the leaf labels are incremented by the depth of $g$, except the rightmost free leaf corresponding to the only edge connected to the dual root vertex. Then the exploration continues unalteredly on other parts of $M_2$. We thus construct $\mathrm{T}(\dual(M))$ from $\mathrm{T}(\dual(M_2))$ and $\mathrm{T}(\dual(\DeltaM(M_1,i)))$, and it agrees with the definition of $\hT$ in this case.
\end{itemize}
By structural induction, we have $\mathrm{T} \circ \dual = \hT \circ \mathrm{T}$.
\end{proof}

\begin{figure}
\centering
\includegraphics[page=17, scale=0.9]{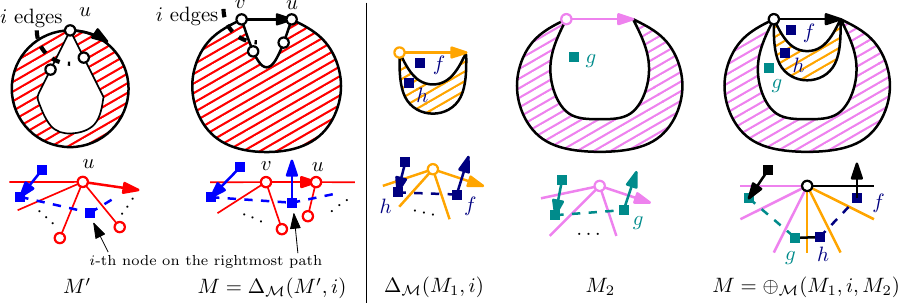}
\caption{Decorated trees of the dual in recursive decomposition of NSP-maps} \label{fig:deco-dual-case}
\end{figure}

According to Proposition~\ref{prop:hT-equiv-h}, the function $\hT$ is the equivalent of $h$ in decorated trees. Therefore, Theorem~\ref{thm:h-involution} and Theorem~\ref{thm:h-autodual} are now natural consequences of Theorem~\ref{thm:duality-equiv-hT}.

\begin{proof}[Proof of Theorem~\ref{thm:h-involution} and Theorem~\ref{thm:h-autodual}]
By Theorem~\ref{thm:duality-equiv-hT} and Proposition~\ref{prop:hT-equiv-h}, we have $\dual = \mathrm{T}^{-1} \circ \phi_T^{-1} \circ h \circ \phi_T \circ \mathrm{T}$. Since $\dual$ is an involution, $h$ is also an involution. For any $\beta$-(1,0) tree $B$, let $M = \mathrm{T}^{-1}(\phi_T^{-1}(B))$, we have $\rpath(h(B)) = \opname{deg}(M) = \flT(\mathrm{T}(M)) = \rootB(B)$. Since the canonical bijections $\phi_T$ and $\mathrm{T}$ conserve the size, the number of fixed points of $h$ in $\mathcal{B}_n$ is equal to the number of fixed points of $\dual$ in $\mathcal{M}_n$, that is, the number of self-dual NSP-maps with $n+1$ edges.
\end{proof}

We thus have a natural bijective proof of Theorem~\ref{thm:h-autodual} as asked in \dinvref{}.

We now deal with the case in the realm of synchronized intervals. We start by a well-known involution on Dyck path that interacts well with the Tamari lattice. Given a non-empty Dyck path $D$, it can be uniquely decomposed into $D = D_1 u D_2 d$, with both $D_1, D_2$ smaller (or empty) Dyck paths, by taking the last non-initial contact as the starting point of $D_2$. We define a transformation $\mir : \mathcal{D} \to \mathcal{D}$ recursively based on this decomposition:
\begin{itemize}
\item If $D$ is empty, then $\mir(D)$ is also empty;
\item Otherwise, suppose that $D = D_1 u D_2 d$, we define $\mir(D) = \mir(D_2) u \mir(D_1) d$.
\end{itemize}
We see that $\mir(D)$ has the same size as $D$. For readers familiar with the Tamari lattice, binary trees can also be used to define the Tamari lattice, with tree rotation that gives the order. The involution $\mir$ on Dyck paths corresponds in fact to taking the reflection by a vertical axis of binary trees, under a certain bijection between Dyck paths and binary trees.

The following properties of $\mir$ is well-known and can be proved inductively using the decomposition $D=D_1 u D_2 d$.

\begin{prop}
  The transformation $\mir$ is an involution on $\mathcal{D}_n$ for all $n$. Furthermore, for two Dyck paths $D_1 \preceq D_2$ in the Tamari lattice, we have $\mir(D_1) \succeq \mir(D_2)$.
\end{prop}

The involution $\mir$ can thus be regarded as an isomorphism between the Tamari lattice $(\mathcal{D}_n, \preceq)$ and its dual $(\mathcal{D}_n, \succeq)$. Furthermore, it also induces an involution on the set of Tamari intervals. By abuse of notation, we denote by $\mir([P,Q])$ the interval $[\mir(Q), \mir(P)]$. The following property is harder to see.

\begin{prop}[Corollary of Proposition~11 and Theorem~2 in \tamref{}]
Let $P,Q$ be two Dyck paths in $\mathcal{D}_n$ of the same type. Then $\mir(P)$ and $\mir(Q)$ are also of the same type. As a consequence, for a synchronized interval $I$, its image $\mir(I)$ is also a synchronized interval.
\end{prop}

Interestingly, the involution $\mir$ that we present here, when restricted to synchronized intervals, is also related to the mysterious involution $h$ of $\beta$-(1,0) trees. In fact, $\mir$ is transferred to $h$ under the conjugation of canonical bijections. Since we have a direct bijection $\mathrm{I}$ from synchronized intervals to decorated trees, we will use decorated trees to relay this equivalence. We start by a lemma on computing the involution $\mir$ on Dyck paths.

\begin{lem} \label{lem:mir-comp}
  Given a Dyck path $P = P_1 u P_2 d u P_3 d \cdots u P_k d$ with $P_1, P_2, \ldots, P_k$ all Dyck paths, we have $\mir(P) = \mir(P_k) u \mir(P_{k-1}) u \cdots u \mir(P_1) d^{k-1}$.
\end{lem}
\begin{proof}
  We prove by induction on $k$. The case $k=2$ comes from the definition of $\mir$. The induction step is easily checked by substituting $D_1 = P_1 u P_2 d \cdots u P_{k-1} d$ and $D_2 = P_k$ into $\mir(D_1 u D_2 d) = \mir(D_2) u \mir(D_1) d$ and apply the induction hypothesis.
\end{proof}

\begin{coro} \label{coro:mir-comp}
  If a Dyck path $P$ has $k$ non-initial contacts, then $\mir(P)$ ends with $k$ consecutive down steps.
\end{coro}
\begin{proof}
  Since $P$ has $k$ non-initial contacts, it can be uniquely written as $P = u P_1 d u P_2 d \cdots u P_k d$ with every $P_i$ a Dyck path. We then conclude by applying Lemma~\ref{lem:mir-comp} with the first part empty.
\end{proof}

\begin{thm} \label{thm:mir-equiv-hT}
  Given a synchronized interval $I$, the decorated trees associated to $I$ and $\mir(I)$ are related by the involution $\hT$. Or equivalently, we have $\mir = \mathrm{I} \circ \hT \circ \mathrm{I}^{-1}$.
\end{thm}
\begin{proof}
  We prove instead that $\mir \circ \mathrm{I} = \mathrm{I} \circ \hT$. We use induction on the size of decorated trees. Let $T$ be a decorated tree, and we suppose that every decorated tree with fewer nodes than $T$ already satisfies $\mir(\mathrm{I}(T)) = \mathrm{I}(\hT(T))$. We define $T^\dagger = \hT(T)$ and $[P^\dagger, Q^\dagger] = \mathrm{I}(T^\dagger)$. We now consider the four types of $T$.

  \begin{itemize}
  \item \textbf{Type I}: This is the base case, which is easily checked.
  \item \textbf{Type II}: We have $T = \DeltaT(T',i)$ for some smaller $T'$ and integer $i$. Suppose that $\mathrm{I}(T')=[P', Q']$, and we can write $P' = P'_1 u P'_2 d \cdots u P'_i d$ with every $P'_i$ a Dyck path. By the induction hypothesis and Lemma~\ref{lem:mir-comp}, we have
    \[ \mathrm{I}(\hT(T')) = \mir(\mathrm{I}(T')) = [\mir(Q'), \mir(P'_i) u \mir(P_{i-1}) \cdots u \mir(P_1) d^{i-1}].\]
    We also have $\mathrm{I}(T) = [ u P'_1 d u P'_2 d \cdots u P'_i d, u Q' d]$, and by Lemma~\ref{lem:mir-comp},
    \[ \mir(\mathrm{I}(T)) = [\mir(Q')ud, \mir(P'_i) u \mir(P_{i-1}) \cdots u \mir(P_1)u d^i].\]

    Now we compute $\mathrm{I}(\hT(T))$. We know that the tree structure of $T^\dagger$ comes from that of $\hT(T)$ by adding a leaf to the right of the $i^{\rm th}$ node on the right-most path of $\hT(T')$. Therefore, $Q^\dagger = \mir(P'_i) u \mir(P_{i-1}) \cdots u \mir(P_1) u d^i$. For $P^\dagger$, we observe that since $\hT(T')$ is a decorated tree, the new leaf added to get $\hT(T)$ is never a certificate, and the certificate-counting function of $\hT(T)$ has the same value on leaves already in $\hT(T')$. Therefore, $P^\dagger = \mir(Q') ud$, where the appended $ud$ is for the added leaf. We thus have $[P^\dagger, Q^\dagger] = \mir(\mathrm{I}(T))$.
  \item \textbf{Type III}: Let $\ell$ be the left-most child of the root of $T$, and $T_1$ the tree obtained by removing $\ell$ from $T$. Suppose that $\mathrm{I}(T_1)=[P_1,Q_1]$. Since $\ell$ cannot be the certificate of any non-root internal node, we have $\mathrm{I}(T) = [udP_1, udQ_1]$. Suppose that $P_1$ has $k$ non-initial contacts and $Q_1$ has $m \leq k$ non-initial contacts, then by Corollary~\ref{coro:mir-comp}, we have $\mir(P_1) = P^\dagger_{1,-} d^k$ and $\mir(Q_1) = Q^\dagger_{1,-} d^m$ for some paths $P^\dagger_{1,-}$ and $Q^\dagger_{1,-}$. By induction hypothesis, we have $\mathrm{I}(\hT(T_1)) = \mir(\mathrm{I}(T_1)) = [ Q^\dagger_{1,-} d^{m}, P^\dagger_{1,-} d^{k}]$. Using Lemma~\ref{lem:mir-comp} with the decomposition in the proof of Corollary~\ref{coro:mir-comp}, we also have $\mir(\mathrm{I}(T)) = [ Q^\dagger_{1,-} u d^{m+1}, P^\dagger_{1,-} u d^{k+1} ]$.
    
    Now for $\mathrm{I}(\hT(T))$, let $\ell'$ be the last leaf of $\hT(T_1)$. We observe that the tree structure of $T^\dagger$ is that of $\hT(T_1)$ with a leaf $\ell''$ added to $\ell'$. Therefore, we have $Q^\dagger = P^\dagger_{1,-} u d^{k+1}$. For $P^\dagger$, since internal nodes in $\hT(T_1)$ with $\ell'$ as certificate will have $\ell''$ as certificate in $T^\dagger$, and the certificates of other non-root internal nodes remain the same. Therefore, we easily check that $P^\dagger = Q^\dagger_{1,-} u d^{m+1}$, and we conclude that $[P^\dagger, Q^\dagger] = \mir(\mathrm{I}(T))$ in this case.
  \item \textbf{Type IV}: We have $T = \oplusT(T_1,i,T_2)$ in this case. Suppose that $\mathrm{I}(T) = [P,Q]$, $\mathrm{I}(T_1) = [P_1, Q_1]$ and $\mathrm{I}(T_2)=[P_2,Q_2]$. As in the type II case, we can write $P_1 = P_{1,1} u P_{1,2} d \cdots u P_{1,i} d$. It is clear that $Q = u Q_1 d Q_2$. For $P$, we observe that the certificate of the left-most child of the root, which was the original root of $T_1$, is the $i^{\rm th}$ free leaf from right to left that comes from $T_1$, while the certificates of other already existing internal nodes remain the same after the fusion. Therefore, we have $P = u P_{1,1} d \cdots u P_{1,i} d P_2$. We suppose that $P_2$ and $Q_2$ have $k$ and $m$ non-initial contacts respectively. By Corollary~\ref{lem:mir-comp}, we have $\mir(P_2) = P^\dagger_{2,-} d^k$ and $\mir(Q_2) = Q^\dagger_{2,-} d^m$ for some paths $P^\dagger_{2,-}, Q^\dagger_{2,-}$. By induction hypothesis, using Lemma~\ref{lem:mir-comp} and Corollary~\ref{lem:mir-comp}, we have
    \begin{align*}
      \mathrm{I}(\hT(T_1)) = \mir(\mathrm{I}(T_1)) &= [\mir(Q_1), \mir(P_{1,i}) u \mir(P_{1,i-1}) \cdots u \mir(P_{1,1}) d^{i-1}], \\
      \mathrm{I}(\hT(T_2)) = \mir(\mathrm{I}(T_2)) &= [Q^\dagger_{2,-} d^m, P^\dagger_{2,-} d^k].
    \end{align*}
    Similarly, using Lemma~\ref{lem:mir-comp} with the decomposition in Corollary~\ref{lem:mir-comp} for the parts of $P_2$ and $Q_2$, we have
    \[ \mir(\mathrm{I}(T)) = [ Q^\dagger_{2,-} \mir(Q_1) u d^{m+1}, P^\dagger_{2,-} \mir(P_{1,i}) u \mir(P_{1,i-1}) \cdots u \mir(P_{1,1}) u d^{i+k}]. \]

    We now consider $\mathrm{I}(\hT(T))$. Let $r$ be the length of the right-most path of $\hT(T_2)$. By the definition of $\hT$, the decorated tree $T^\dagger$ is constructed by first adding $r$ to the labels of $\hT(T_1)$, then adding a free leaf $\ell$ to the $i^{\rm th}$ node on the right-most path, and finally attaching the tree to the right-most leaf $\ell'$ of $\hT(T_2)$. Therefore, for the path $Q^\dagger$, we have $Q^\dagger = P^\dagger_{2,-} \mir(P_1,i) u \mir(P_{1,i-1}) \cdots u \mir(P_{1,1}) u d^{i+k}$. For the path $P^\dagger$, we now consider certificates of internal nodes of $\hT(T)$. Since $\hT(T_1)$ is already a decorated tree, the new free leaf is never a certificate of nodes in $\hT(T_1)$. Due to the increase by $r$ of labels in the part from $\hT(T_1)$, leaves in the part from $\hT(T_1)$ other than $\ell$ cannot be the certificate of any node in the part from $\hT(T_2)$, and the internal node corresponding to $\ell'$ has $\ell$ as certificate. Since the depth of every node originated from $\hT(T_1)$ also increases by $r$ as the labels, the certificate of a node originated from a non-root node of $\hT(T_1)$ stays the same in $\hT(T)$. The only change of certificates is for nodes in $T_2$ whose certificate was $\ell'$, and it is clear that their certificate in $\hT(T)$ is now $\ell$. We thus easily check that $P^\dagger = Q^\dagger_{2,-} \mir(Q_1) u d^{m+1}$. We then conclude that $[P^\dagger, Q^\dagger] = \mir(\mathrm{I}(T))$.
  \end{itemize}

  With all cases checked, by induction on the size of $T$, we conclude that $\mir \circ \mathrm{I} = \mathrm{I} \circ \hT$, which implies $\mir = \mathrm{I} \circ \hT \circ \mathrm{I}^{-1}$.
\end{proof}

As a corollary of Theorem~\ref{thm:mir-equiv-hT}, Theorem~\ref{thm:duality-equiv-hT} and Theorem~\ref{thm:T-recursive-canonical}, we now know that, under the bijection described in \tambijref{} from NSP-maps and synchronized intervals, which is just $\mathrm{I} \circ \mathrm{T}$, the map duality is transferred to the involution of synchronized intervals induced by the classic isomorphism of the Tamari lattice $(\mathcal{D}_n, \preceq)$ to its order dual $(\mathcal{D}_n, \succeq)$. Furthermore, we have the following result that is similar to Theorem~\ref{thm:h-autodual} on the number of fixed points of $\mir$.

\begin{thm} \label{thm:fixed-points-mir}
  The number of fixed points of $\mir$ in $\mathcal{I}_n$, that is, the number of synchronized intervals of the form $[P, \mir(P)]$, is equal to the number of self-dual NSP-maps in $\mathcal{M}_n$.
\end{thm}

\section{Series decomposition and duality} \label{sec:series}

In this section, we will investigate the link between the series decomposition of NSP-maps and map duality.

The series decomposition of NSP-maps was first introduced by Tutte in \tutte{} and used in \cite{desc-tree, JS1998bijective}, which consists of looking at how \emph{deletion} of the root would change the map. We now describe an adapted version of the series decomposition. Let $M$ be an NSP-maps rooted at $e$ pointing from $v$ to $u$, and $M'$ the map $M$ with $e$ deleted. There might be several cut vertices in $M'$, and we can split $M'$ into non-separable components by cutting all these cut vertices. The components we get can be ordered, starting from the component containing $v$, and the last component is the one containing $u$. Each component is either an edge or an NSP-map. $M$ must fall into one of the following types, as illustrated in Figure~\ref{fig:series-cases}.

\begin{itemize}
\item \textbf{Type I}: $M'$ has only one component, and it is a single edge.
\item \textbf{Type II}: $M'$ has only one component, and it is an NSP-map.
\item \textbf{Type III}: $M'$ has at least two components, and the first is an edge.
\item \textbf{Type IV}: $M'$ has at least two components, and the first is an NSP-map.
\end{itemize}

\begin{figure}
  \centering
  \includegraphics[page=22, scale=0.8]{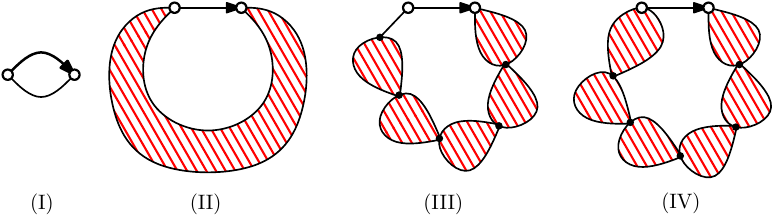}
  \caption{Four cases of the series decomposition of NSP-maps}
  \label{fig:series-cases}
\end{figure}

As in parallel decomposition, in series decomposition we consider the single edge as a non-separable component. We also introduce another statistic \tdef{$\fdeg$} of NSP-maps, which is the degree of the outer face minus 1, where the excluded 1 stands for the root. We have the following relation between the two statistics $\opname{deg}$ and $\fdeg$.

\begin{prop} \label{prop:deg-fdeg}
For any NSP-map $M$, we have $\fdeg(M) = \opname{deg}(\dual(M))$.
\end{prop}
\begin{proof}
It is because the root vertex of $\dual(M)$ is the dual vertex of the outer face of $M$.
\end{proof}

We will now define some operations on NSP-maps used in the series decomposition, illustrated in Figure~\ref{fig:series-fct}. Let $M$ be an NSP-map rooted at $e = (v \to u)$. We denote by $\PiS(M)$ the map obtained by deleting $e$ from $M$ and appointing the edge on the new outer face that points to $u$ in clockwise order as the root. For any $1 \leq i \leq \fdeg(M)$, let $u_i$ be the $i$-th vertex on the outer face in clockwise order, with $u_0 = u$. We denote by $\DeltaS(M, i)$ the map obtained by adding an edge $e'$ in $M$ from $u_i$ to $u$ such that the new outer face is adjacent to all $u_j$ with $j \leq i$, and the new root is $e'$ pointing from $u_i$ to $u$. The following proposition is clear.

\begin{prop} \label{prop:pi-delta-s}
Let $M$ be an NSP-map, $i$ an integer between $1$ and $\fdeg(M)$ and $M' = \DeltaS(M,i)$. We have $\fdeg(M) = i$ and $\PiS(M') = M$, and $M'$ is also an NSP-map. Furthermore, for any NSP-map $N$ such that $\PiS(N) = M$, there is a unique integer $j$ between $1$ and $\fdeg(M)$ such that $N = \DeltaS(M,j)$.
\end{prop}

Let $M_1, M_2$ be two NSP-maps, and $i$ an integer between $1$ and $\fdeg(M_1)$. Let $e=(v \to u)$ be the root of $M_1$, and $u_i$ the $i$-th vertex on the outer face in clockwise order, with $u_0 = u$. We denote by $\oplusS(M_1, i, M_2)$ the map obtained by first splitting the root vertex $v'$ of $M_2$ into $v'_1$ attached only to the root and $v'_2$ attached to other adjacent edges, then identifying $v'_1$ with $u_i$ and $v'_2$ with $u$. Inversely, given an NSP-map $M$ of type IV in the series decomposition, we define $\hdS(M)$ to be the first component of $M$, rooted at its first edge in clockwise order adjacent to the core face and pointing in the direction such that the core face is on the left, and $\tlS(M)$ the map that remains after removing $\hdS(M)$ from $M$ and identifying the root vertex of $M$ and the cut vertex associated to $\hdS(M)$. Both $\hdS(M)$ and $\tlS(M)$ are NSP-maps, and the following properties are clear from construction.

\begin{figure}
  \centering
  \includegraphics[page=23,width=\textwidth]{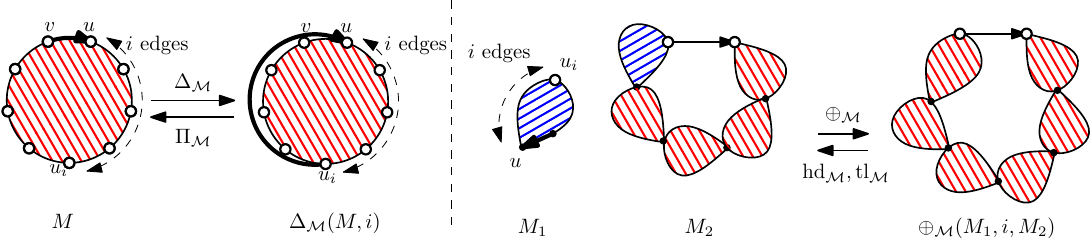}
  \caption{Functions related to the series decomposition of NSP-maps}
  \label{fig:series-fct}
\end{figure}

\begin{prop} \label{prop:oplus-hd-tl-s}
For two NSP-maps $M_1, M_2$ and an integer $i$ between $1$ and $\fdeg(M_1)$, the map $M = \oplusS(M_1, i, M_2)$ is also an NSP-map, and we have $\hdM(M) = M_1$, $\tlM(M) = M_2$, and $\fdeg(M) = i + \fdeg(M_2)$.

Furthermore, for any NSP-map $M'$ of type IV in the series decomposition, if $\hdM(M') = M_1$ and $\tlM(M') = M_2$, then there must exist an integer $j$ between $1$ and $\fdeg(M_1)$ such that $M' = \oplusM(M_1, j, M_2)$.
\end{prop}

By Proposition~\ref{prop:pi-delta-s} and Proposition~\ref{prop:oplus-hd-tl-s}, we have the following series decomposition.

\begin{itemize}
\item \textbf{Type I}: This is the base case.
\item \textbf{Type II}: In this case, the function $\PiS$ applies, and all NSP-map in this case can be obtained bijectively by applying $\DeltaS$ to NSP-maps with all possible extra parameter.
\item \textbf{Type III}: Let $e_1$ be the edge of the last component. By contracting $e_1$ we get a smaller NSP-map. To obtain an NSP-map in this case, we take an arbitrary NSP-map $M'$ rooted at $e$ and add a vertex in the middle, and we have $\fdeg(M) = \fdeg(M') + 1$.
\item \textbf{Type IV}: In this case, the functions $\hdS$ and $\tlS$ apply, and all NSP-map in this case can be constructed by applying $\oplusS$ to NSP-maps with all possible extra parameter.
\end{itemize}

Similar to the parallel decomposition, the series decomposition also has the same recursive construction as that of $\beta$-(1,0) tree with coinciding size parameters. The statistics $\fdeg$ and $\rootB$ also coincide. Therefore, by identifying all four cases, we can also define the canonical map $\phi_S$ from the set of NSP-map $\mathcal{M}$ to the set of $\beta$-(1,0) trees $\mathcal{B}$.

We now introduce another function $\NR$ on NSP-maps. Let $M$ be an NSP-map rooted at $e$. By re-rooting $M$ to the edge next to $e$ in clockwise order without changing the outer face, we get another NSP-map, denoted by $\NR(M)$. We have $\fdeg(\NR(M)) = \fdeg(M)$. We can now announce the following theorem that relates the series decomposition to the parallel decomposition.

\begin{figure}
  \centering
  \includegraphics[scale=0.85,page=24]{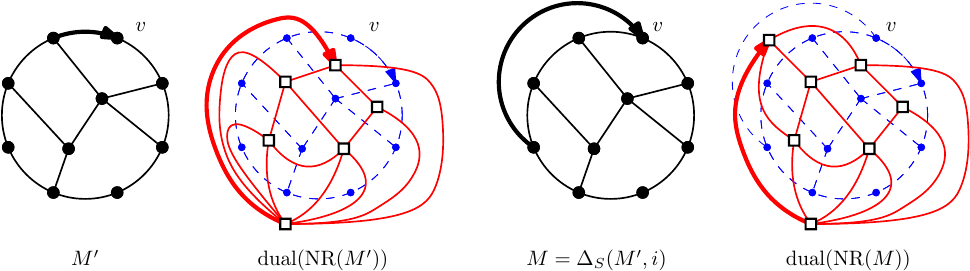}
  \caption{Example of the Type II case of Theorem~\ref{thm:series-to-parallel}}
  \label{fig:s-to-p-II}
\end{figure}

\begin{figure}
  \centering
  \includegraphics[page=25,scale=0.85]{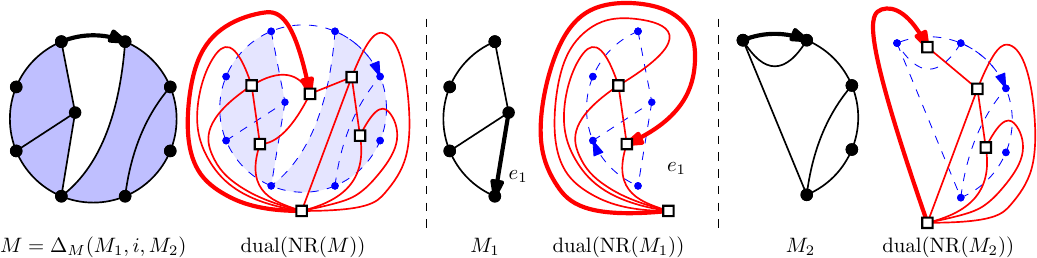}
  \caption{Example of the Type IV case of Theorem~\ref{thm:series-to-parallel}}
  \label{fig:s-to-p-IV}
\end{figure}

\begin{thm} \label{thm:series-to-parallel}
We have $\phi_S = \phi_M \circ \dual \circ \NR$, where $\phi_S$ (\emph{resp.} $\phi_M$) is the canonical bijection from NSP-maps to $\beta$-(1,0) trees with respect to the series (\emph{resp.} parallel) decomposition.
\end{thm}
\begin{proof}
We proceed by induction on NSP-maps with size as induction parameter. Let $M$ be an NSP-map, and we suppose that every map $M'$ with fewer edges than $M$ satisfies $\phi_S(M') = \phi_M(\dual(\NR(M')))$. We define $M_\to = \NR(M)$. Let $e$ be the root of $M$ and $e'$ the root of $M_\to$. Therefore, the root of $\dual(M_\to)$ is the dual edge of $e$. We now consider the possible type of $M$ under the \textbf{series} decomposition.
\begin{itemize}
\item \textbf{Type I}: This is the base case, and is easily verified.
\item \textbf{Type II}: In this case, we have $M = \DeltaS(M',i)$ with $M'$ an NSP-map and $i$ an integer such that $1 \leq i \leq \fdeg{M'}$. An example of this case is given in Figure~\ref{fig:s-to-p-II}. By induction hypothesis, we only need to prove that $\DeltaM(\dual(\NR(M')),i) = \dual(M_\to)$. We first observe that $\opname{deg}(\dual(M_\to)) = \fdeg(M) = i$. Therefore, by Proposition~\ref{prop:pi-delta-m}, we only need to prove that $\dual(\NR(M')) = \PiM(\dual(M_\to))$. Suppose that the root of $M'$ points from $u$ to $v$. We observe that the root vertices of both $\NR(M')$ and $M_\to$ inherit from $v$ in $M'$, so their dual share the same outer face. By applying $\PiM$, we contract the root of $\dual(M_\to)$, which is the dual of $e$. Since the contraction of an edge is the deletion of its dual in the dual map, and by deleting $e$ in $M_\to$ we get $\NR(M')$, which concludes this case.
\item \textbf{Type III}: In this case, $M$ is obtained by adding a vertex in the middle of the root of another NSP-map $M'$. By induction hypothesis, we only need to prove that $\phi_M(\dual(M_\to))$ is obtained by adding a leaf as the leftmost direct sub-tree to $\phi_M(\dual(\NR(M')))$. But this is clear since adding a vertex in the middle of an edge is duplicating the dual edge in the dual map.
\item \textbf{Type IV}: In this case, we have $M = \oplusS(M_1, i, M_2)$ for some NSP-maps $M_1, M_2$ and an integer $i$ with $1 \leq i \leq \fdeg{M_1}$. An example of this case is given in Figure~\ref{fig:s-to-p-IV}. By induction hypothesis, we only need to prove that $\oplusM(\dual(\NR(M_1)), i, \dual(\NR(M_2))) = \dual(M_\to)$. Since $\opname{deg}(\dual(M_\to)) = \fdeg(M) = i + \fdeg(M_2) = i + \opname{deg}(\dual(M_2))$, by Proposition~\ref{prop:oplus-hd-tl-m}, we only need to prove that $\hdM(\dual(M_\to)) = \dual(\NR(M_1))$ and $\tlM(\dual(M_\to)) = \dual(\NR(M_2))$. We first observe that the dual faces of cut vertices in $M_\to$ stand between non-separable components in $\dual(M_\to)$. Let $w$ be the cut vertex in $M$ whose dual face stands between the innermost component of $\dual(M_\to)$ and the other components, and $e_1$ be the root of $M_1$. The map $\dual(\NR(M_1))$ is thus rooted at the dual edge of $e_1$. By the construction of $\hdM$, the dual edge of $e_1$ is also the root of $\hdM(\dual(M_\to))$. Therefore, we have the first equality. The second follows by a similar verification of root.
\end{itemize}
We thus complete the induction.
\end{proof}

\section{Transfer of statistics} \label{sec:stats}

Using mechanisms set up previously, we can easily prove several results of statistic transfer between different structures. The following list consists of natural bijections that we have introduced.

\begin{center}
\begin{tabular}{cccc}
  $\mathcal{M} \to \mathcal{B}$: $\phi_M, \phi_S$ &
  $\mathcal{M} \to \mathcal{T}$: $\mathrm{T}$ &
  $\mathcal{T} \to \mathcal{B}$: $\phi_T$ &
  $\mathcal{T} \to \mathcal{I}$: $\mathrm{I}$
  \\
  $\mathcal{M} \to \mathcal{M}$: $\dual, \NR$ &
  $\mathcal{B} \to \mathcal{B}$: $h$ &
  $\mathcal{T} \to \mathcal{T}$: $\hT$ &
  $\mathcal{I} \to \mathcal{I}$: $\mir$
\end{tabular}
\end{center}

We now discuss some statistics of all classes $\mathcal{B}, \mathcal{T}, \mathcal{M}, \mathcal{I}$. We start by the tree classes $\mathcal{B}$ and $\mathcal{T}$, with examples given in Figure~\ref{fig:tree-stat}. The following statistics on $\mathcal{B}$ come from \kitaev{}, but also apply to $\mathcal{T}$. Let $T$ be a tree (either a $\beta$-(1,0) tree or a decorated tree), we denote by \tdef{$\opname{leaf}(T)$} the number of leaves in $T$, by \tdef{$\opname{int}(T)$} the number of internal nodes of $T$, by \tdef{$\opname{sub}(T)$} the number of direct children of the root of $T$, by \tdef{$\opname{rpath}(T)$} the length of the rightmost path of $T$, and by \tdef{$\opname{stem}(T)$} the length of the longest path from the root whose intermediate nodes have only one descendant.

\begin{figure}[bhtp]
  \centering
  \includegraphics[page=26,scale=0.8]{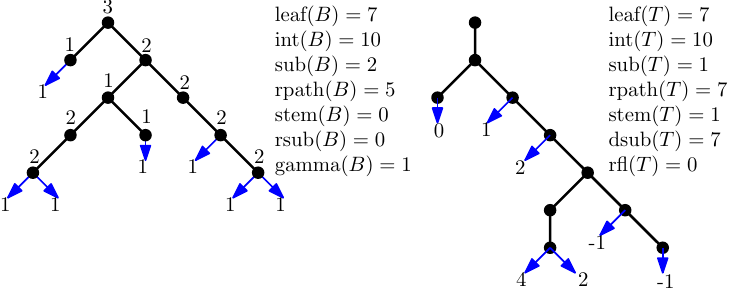}
  \caption{Examples of statistics of $\mathcal{B}$ and $\mathcal{T}$}
  \label{fig:tree-stat}
\end{figure}

We now define some statistics from \kitaev{} defined only on $\beta$-(1,0) trees. Let $B \in \mathcal{B}$ a $\beta$-(1,0) tree. We recall that $\rootB(B)$ is the label of the root, and we denote by \tdef{$\opname{rsub}(B)$} the number of nodes with a label $1$ on the rightmost path (including leaves). Now let $\ell_1, \ell_2, \ldots, \ell_k$ be the leaves of $B$ \emph{from right to left}. Starting from $\ell_1$, we decrease by one the labels of nodes on the path from $\ell_i$ to the root, until one of the internal nodes has a label $1$. The number of leaves proceeded is defined as \tdef{$\opname{gamma}(B)$}.

There are also corresponding statistics that only apply to decorated trees. Let $T \in \mathcal{T}$ a decorated tree. We recall that $\flT(T)$ is the number of free leaves in $T$. We denote by \tdef{$\opname{dsub}(T)$} the number of internal nodes with only one certificate, \textit{i.e.}, descendant leaf whose label is at most the depth of the internal node minus 2. We define \tdef{$\opname{rfl}(T)$} the maximal number of free leaves we can remove from right to left while keeping the decorated tree conditions without creating new leaves.

We now list some statistics on NSP-maps. Let $M \in \mathcal{M}$ be an NSP-map. We recall that $\opname{deg}(M)$ is the degree of the root vertex of $M$ \emph{minus 1}, and $\fdeg(M)$ is the degree of the outer face of $M$ minus $1$. We denote by \tdef{$\opname{face}(M)$} the number of faces of $M$, by \tdef{$\opname{vertex}(M)$} the number of vertices of $M$, by \tdef{$\opname{pcomp}(M)$} the number of components in the parallel decomposition of $M$, by \tdef{$\opname{scomp}(M)$} the number of components in the series decomposition of $M$. Furthermore, we define \tdef{$\opname{ejoin}(M)$} as the maximal number of successive root contraction we can do while leaving an NSP-map (the new root after contraction is the edge next to the original root in clockwise order). We also define \tdef{$\opname{ecut}(M)$} as the maximal number of successive deletion of edges adjacent to the root vertex we can perform while leaving an NSP-map, starting from the edge before the root in clockwise order, and turning in counterclockwise direction.

Finally, we list some statistics of synchronized intervals. Let $I=[P,Q] \in \mathcal{I}$ a synchronized interval. We recall that $\contI(I)$ is the number of non-initial contacts of the lower path $P$, and we denote by \tdef{$\contI^*(I)$} the number of non-initial contacts of the upper path $Q$. We denote by \tdef{$\opname{peak}(I)$} the number of peaks of $P$ (or $Q$, since $I$ is a synchronized interval, $P$ and $Q$ have the same number of peaks), and by \tdef{$\opname{dblu}(I)$} the number of up steps followed by an up step in $P$ (or $Q$, again by the fact that $I$ is synchronized). It is clear that the $\opname{peak}(I)$ and $\opname{dblu}(I)$ are respectively the number of $N$ and $E$ in the type of $P$. We define \tdef{$\opname{desc}(I)$} as the length of the last descent of $Q$, and \tdef{$\opname{level}(I)$} the maximal number $k$ such that $Q$ can be written in the form of $u^k Q' d^k$ where $Q'$ is a Dyck path. The last definition is \tdef{$\opname{teeth}(I)$}, which is the number of repetitions of $ud$ that we can delete at the end of $P$ without making $P$ empty.

We now investigate how our bijections transfer statistics. Suppose that we have two combinatorial classes $\mathcal{A}$ and $\mathcal{C}$ with some size parameter, and we denote by $\mathcal{A}_n$ and $\mathcal{C}_n$ the subset of $\mathcal{A}$ and $\mathcal{C}$ of size $n$ respectively. Furthermore, we suppose that there is a bijection $f$ from $\mathcal{A}$ to $\mathcal{C}$ such that $f(\mathcal{A}_n) = \mathcal{C}_n$ for all $n \geq 0$. Let $s_{\mathcal{A}}$ (resp. $s_{\mathcal{C}}$) a statistic of $\mathcal{A}$ (resp. of $\mathcal{C}$), \textit{i.e.} a function on $\mathcal{A}$ (resp. $\mathcal{C}$). We say that the statistic $s_{\mathcal{A}}$ is \tdef{transferred} to $s_{\mathcal{C}}$ by $f$ if for every $a \in \mathcal{A}$, we have $s_{\mathcal{A}}(a) = s_{\mathcal{C}}(f(a))$. Statistics transferred by a bijection between two combinatorial classes illustrate some of the combinatorial structure preserved by the bijection.

By Theorem~\ref{thm:T-recursive-canonical}, Theorem~\ref{thm:duality-equiv-hT}, Proposition~\ref{prop:hT-equiv-h}, Theorem~\ref{thm:mir-equiv-hT} and Theorem~\ref{thm:series-to-parallel}, we only need to consider $\phi_T$, $\mathrm{T}$, $\dual$ and $\mathrm{NR}$ for statistic transfers. In the following, we will only consider $\phi_T$, $\mathrm{T}$, $\dual$.

\begin{prop} \label{prop:stat-phi-T}
The function $\phi_T$ transfers the list of statistics $[\flT, \opname{1 + dsub}, \opname{rfl}]$ in $\mathcal{T}$ to the statistics $[\rootB, \opname{rsub}, \opname{gamma}]$ in $\mathcal{B}$.
\end{prop}
\begin{proof}
The first pair $(\flT \to \rootB)$ comes directly from the recursive definition. The second pair $(1 + \opname{dsub} \to \opname{rsub})$ comes directly from the definition of the bijection $\varphi_T$, and the rightmost leaf is not counted in $\opname{dsub}$. We now deal with the third pair $(\opname{rfl} \to \opname{gamma})$.

Let $T$ be a decorated tree and $B = \phi_T(T)$ the corresponding $\beta$-(1,0) tree. For simplicity, we will identify the tree structure of $T$ and $B$. Suppose that $\opname{gamma}(B) = k$, and let $\ell_1, \ldots, \ell_k$ be the first leaves  in clockwise order that are proceeded in the computation of $\opname{gamma}(B)$. We now show that all $\ell_i$ are free leaves in $T$. Suppose otherwise that $\ell_i$ is the first non-free leaf among $\ell_1, \ldots, \ell_k$, and let $p \geq 0$ be its label. Let $u$ be the ancestor of $\ell_i$ of depth $p+1$. By the last condition of decorated tree, $\ell_i$ is the last leaf in clockwise order with a label at most $p$ in the sub-tree induced by $u$. However, in the definition of $\varphi_T$, the label of $u$ in $B$ is contributed by leaves with labels at most $p-1$, which are among $\ell_1, \ldots, \ell_{i-1}$. Therefore, before processing $\ell_i$, the label of $u$ should already be decremented to $0$, which violates the definition of $\opname{gamma}$. Thus, all $\ell_i$ must be free, and by the definition of $\opname{gamma}$ and $\varphi_T$, we have $\opname{gamma}(B) = \opname{rfl}(T)$.
\end{proof}

\begin{prop} \label{prop:stat-T}
  The functions $\mathrm{T}$ and $\mathrm{I}$ transfers the following lists of statistics:
\begin{center}
\begin{tabular}{rrrrrrrr}
  In $\mathcal{M}$: \quad & $\opname{deg}$, & $\opname{face}$, & $\opname{vertex}$, & $\opname{pcomp}$, & $\opname{fdeg}$, & $\opname{ejoin}$, & $\opname{ecut}$; \\
  In $\mathcal{T}$: \quad & $\flT$, & $1 + \opname{leaf}$, & $1 + \opname{int}$, & $\opname{sub}$, & $\opname{rpath}$, & $\opname{stem}$, & $\opname{rfl}$; \\
  In $\mathcal{I}$: \quad & $\contI$, & $1 + \opname{peak}$, & $1 + \opname{dblu}$, & $\contI^*$, & $\opname{desc}$, & $\opname{level}$, & $\opname{teeth}$.
\end{tabular}
\end{center}
More precisely, the function $\mathrm{T}$ transfers the first list of statistics in $\mathcal{M}$ to the second one in $\mathcal{T}$, while the function $\mathrm{I}$ transfers the second list of statistics in $\mathcal{T}$ to the third one in $\mathcal{I}$.
\end{prop}
\begin{proof}
  For the first two lines, all pairs but the last one are clear from the definition of $\mathrm{T}$ and the parallel decomposition. For the last pair, let $M$ be an NSP-map and $T = \mathrm{T}(M)$, and we see that the free leaves we remove in the computation of $\opname{rfl}(T)$ corresponds to edges we remove from the root vertex when computing $\opname{ecut}(M)$.

  For the last two lines, every pair except the first and the last one can be seen from the fact that the upper path of a synchronized interval determines the tree structure of its decorated tree. The first pair can easily be seen from the definition of $\mathrm{I}$. For the last pair, we consider a synchronized interval $I = [P,Q]$ and its corresponding decorated tree $T$. Suppose that there are $k$ repetitions of $ud$ at the end of $P$. By the definition of $\mathrm{I}$, These repetitions correspond to the last $k$ leaves $\ell_1, \ldots, \ell_k$ in traversal order of $T$, which are free, and they must be preceded by a series of down steps of length at least two (or one if $P$ is of the form $(ud)^n$), which corresponds to a free leaf $\ell$ in $T$ that is the first child of a node $v$, which is the last explored node. Therefore, all $\ell_i$'s can be deleted while keeping $T$ a decorated tree. However, $\ell$ itself cannot be deleted, since $\ell$ is the first child of $v$. We thus have $\opname{teeth}(I) = \opname{rfl}(T)$.
\end{proof}

\begin{prop} \label{prop:stat-dsub}
For $M$ an NSP-map, $M' = \dual(M)$ and $T' = \mathrm{T}(M')$, we have $\opname{dsub}(T') + 1 = \opname{pcomp}(M)$.
\end{prop}
\begin{proof}
Let $f$ be a face in $M$ such that the dual vertex $u_f$ is on the rightmost path of $T'$, and we suppose that the rightmost child $v$ of $u_f$ is an internal node. Let $p$ be the depth of $v$. The vertex $v$ always has a descendant leaf corresponding to the dual edge of the root of $M$. By the exploration process of $\mathrm{T}$, the vertex $v$ has no other descendant leaf with a label at most $p-2$ if and only if by deleting $u_f$ in $\dual(M)$, the dual edge of the root will become a cut edge, which means that $f$ separates two components in the parallel decomposition of $M$. We conclude the proof by observing that there are $(\opname{pcomp}(M)-1)$ such faces.
\end{proof}

\begin{prop} \label{prop:stat-dual}
The involution $\dual$ exchanges the following pairs of statistics in $\mathcal{M}$: $(\opname{deg}, \opname{fdeg})$, $(\opname{faces}, \opname{vertices})$, $(\opname{ejoin}, \opname{ecut})$. 
\end{prop}
\begin{proof}
The first two pairs are consequence of map duality. For the third one, it comes from the fact that contracting an edge in a planar map is equivalent to deleting the dual edge in its dual.
\end{proof}

We thus have the following result in \dnsref{} on statistic transfers by the involution $h$.
\begin{coro}[Theorem~10 in \dnsref{}]
The involution $h$ exchanges the following pairs of statistics in $\mathcal{B}$: $(\opname{leaves}, \opname{int})$, $(\opname{root}, \opname{rpath})$, $(\opname{gamma}, \opname{stem})$ and $(\opname{sub}, \opname{rsub})$.
\end{coro}
\begin{proof}
All pairs are based on the conjugate relations between $\dual$, $h$ and $\hT$ in Theorem~\ref{thm:duality-equiv-hT} and Proposition~\ref{prop:hT-equiv-h}, and the statistic transfers between different classes of objects in Proposition~\ref{prop:stat-phi-T} and Proposition~\ref{prop:stat-T}. On these bases, the first three pairs comes from Proposition~\ref{prop:stat-dual}, and the last one comes from Proposition~\ref{prop:stat-dsub}.
\end{proof}

Similar results on statistics on Tamari intervals are omitted here, since they are well-known.

\section*{Acknowledgement}

The author thanks Mireille Bousquet-Mélou, Guillaume Chapuy, Sergey Kitaev and Louis-François Préville-Ratelle for their useful comments and inspiring discussions. The present work started under the hospitality of Université de Bordeaux, and finishes under the present postdoctoral position that the author occupies, financed by LIP of ENS de Lyon, ANR Grant STINT (ANR-13-BS02-0007) and Project-Team AriC of INRIA.

\bibliographystyle{alpha}
\bibliography{involution}

\begin{thebibliography}{BMFPR11}

\bibitem[BB09]{BB2009intervals}
O.~Bernardi and N.~Bonichon.
\newblock Intervals in {C}atalan lattices and realizers of triangulations.
\newblock {\em J. Combin. Theory Ser. A}, 116(1):55--75, 2009.

\bibitem[BMFPR11]{bousquet-fusy-preville}
M.~Bousquet-M\'elou, \'E. Fusy, and L.-F. Pr\'eville-Ratelle.
\newblock The number of intervals in the $m$-{T}amari lattices.
\newblock {\em Electron. J. Combin.}, 18(2):Research Paper 31, 26 pp.
  (electronic), 2011.

\bibitem[CCP14]{chapoton-chatel-pons}
F.~Chapoton, G.~Ch{\^a}tel, and V.~Pons.
\newblock Two bijections on {T}amari intervals.
\newblock In {\em 26th {I}nternational {C}onference on {F}ormal {P}ower
  {S}eries and {A}lgebraic {C}ombinatorics ({FPSAC} 2014)}, Discrete Math.
  Theor. Comput. Sci. Proc., AT, pages 241--252. Assoc. Discrete Math. Theor.
  Comput. Sci., Nancy, 2014.
\newblock arXiv:1311.4382.

\bibitem[Cha06]{ch06}
F.~Chapoton.
\newblock Sur le nombre d'intervalles dans les treillis de {T}amari.
\newblock {\em S\'em. Lothar. Combin.}, pages Art. B55f, 18 pp. (electronic),
  2006.

\bibitem[CKS09]{kitaev2009decomp}
A.~Claesson, S.~Kitaev, and E.~Steingr{\' i}msson.
\newblock Decompositions and statistics for {$\beta(1,0)$}-trees and
  nonseparable permutations.
\newblock {\em Adv. in Appl. Math.}, 42(3):313--328, 2009.

\bibitem[CKS13]{kitaev-involution}
A.~Claesson, S.~Kitaev, and E.~Steingr{\'{\i}}msson.
\newblock An involution on {$\beta(1,0)$}-trees.
\newblock {\em Adv. in Appl. Math.}, 51(2):276--284, 2013.

\bibitem[CP13]{interval-poset}
G.~Ch{\^a}tel and V.~Pons.
\newblock Counting smaller trees in the {T}amari order.
\newblock In {\em 25th {I}nternational {C}onference on {F}ormal {P}ower
  {S}eries and {A}lgebraic {C}ombinatorics ({FPSAC} 2013)}, Discrete Math.
  Theor. Comput. Sci. Proc., AS, pages 433--444. Assoc. Discrete Math. Theor.
  Comput. Sci., Nancy, 2013.
\newblock arXiv:1212.0751.

\bibitem[CS03]{desc-tree}
R.~Cori and G.~Schaeffer.
\newblock Description trees and {T}utte formulas.
\newblock {\em Theoret. Comput. Sci.}, 292(1):165--183, 2003.
\newblock Selected papers in honor of Jean Berstel.

\bibitem[DGW96]{dulucq1993permutations}
S.~Dulucq, S.~Gire, and J.~West.
\newblock Permutations with forbidden subsequences and nonseparable planar
  maps.
\newblock In {\em Proceedings of the 5th {C}onference on {F}ormal {P}ower
  {S}eries and {A}lgebraic {C}ombinatorics ({F}lorence, 1993)}, volume 153,
  pages 85--103, 1996.

\bibitem[FPR17]{tam-non-sep}
W.~Fang and L.-F. Préville-Ratelle.
\newblock The enumeration of generalized {T}amari intervals.
\newblock {\em European J. Combin.}, 61:69--84, 2017.

\bibitem[JS98]{JS1998bijective}
B.~Jacquard and G.~Schaeffer.
\newblock A bijective census of nonseparable planar maps.
\newblock {\em J. Combin. Theory Ser. A}, 83(1):1--20, 1998.

\bibitem[KdM13]{kitaev-nsp}
S.~Kitaev and A.~de~Mier.
\newblock Enumeration of fixed points of an involution on $\beta$(1, 0)-trees.
\newblock {\em Graphs and Combinatorics}, 30(5):1207--1221, 2013.

\bibitem[KdMN14]{count-self-dual}
S.~Kitaev, A.~de~Mier, and M.~Noy.
\newblock On the number of self-dual rooted maps.
\newblock {\em European J. Combin.}, 35:377--387, 2014.

\bibitem[Kit11]{kitaev-book}
S.~Kitaev.
\newblock {\em Patterns in permutations and words}.
\newblock Monographs in Theoretical Computer Science. An EATCS Series.
  Springer, Heidelberg, 2011.

\bibitem[PRVar]{PRV2014extension}
L.-F. Pr{\'e}ville-Ratelle and X.~Viennot.
\newblock An extension of {T}amari lattices.
\newblock {\em Transactions of the AMS}, To appear.
\newblock arXiv:1406.3787.

\bibitem[Tut63]{Tutte:census}
W.~T. Tutte.
\newblock A census of planar maps.
\newblock {\em Canad. J. Math.}, 15:249--271, 1963.

\end{thebibliography}

\end{document}